\documentclass{amsart}
\usepackage[utf8]{inputenc}
\usepackage[english]{babel}
\usepackage[all]{xy}
\usepackage{graphicx}
\usepackage{amsmath,amsthm,bm,mathrsfs,amscd,tikz}
\usepackage{amsfonts,amssymb,mathrsfs}
\usepackage{verbatim}
\usepackage{float}
\usepackage{enumerate}
\RequirePackage{color}
\usepackage{xcolor}
\usepackage{tikz}
\usetikzlibrary{arrows.meta}
\usetikzlibrary{patterns}
\usetikzlibrary{matrix, arrows, decorations.pathmorphing}
\usepackage{soul}
\usepackage[pdftex,colorlinks,linkcolor  = black,citecolor=blue]{hyperref}
\usepackage{pdfpages}
\usepackage{tabu}

\theoremstyle{plain}
\newtheorem{theorem}{Theorem}[section]
\newtheorem{proposition}[theorem]{Proposition}
\newtheorem{corollary}[theorem]{Corollary}
\newtheorem{lemma}[theorem]{Lemma}
\newtheorem{conjecture}[theorem]{Conjecture}
\newtheorem{remark}[theorem]{Remark}

\definecolor{job}{RGB}{200,65,0}

\definecolor{bettergreen}{RGB}{0,155,0}

\newtheorem{thm}{Theorem}

\theoremstyle{definition}
\newtheorem{definition}[theorem]{Definition}
\newtheorem{example}[theorem]{Example}

\newcommand{\Ext}{\operatorname{Ext}}
\newcommand{\Hom}{\operatorname{Hom}}
\newcommand{\supp}{\operatorname{supp}}
\newcommand{\Span}{\operatorname{span}}

\newcommand{\Rep}{{\rm Rep}}

\newcommand{\End}{{\rm End}}
\newcommand{\mmod}{{\rm mod}\,}

\newcommand{\rpwf}{{\rm rep}^{\rm pwf}}

\newcommand{\rfp}{{\rm rep}^{\rm fp}}
\newcommand{\rqnf}{{\rm rep}^{\rm qnf}}
\newcommand{\Mod }{{\rm Mod}}

\newcommand{\C}{\mathcal{C}}
\newcommand{\M}{\mathcal{M}}
\newcommand{\I}{\mathcal{I}}

\newcommand{\ResM}{\mathrm{Res}_{\mathcal{M}}}
\newcommand{\IndP}{\mathrm{Ind}_{\mathfrak{P}}}
\newcommand{\IndPM}{\mathrm{Ind}_{\mathfrak{P}_M}}
\newcommand{\Cbar}{\mathcal{A}}
\newcommand{\Pbar}{\overline{\mathcal{P}}}
\newcommand{\Chat}{\overline{\mathcal{A}}^{\mathfrak{P}}}
\newcommand{\Ihat}{\overline{\mathcal{I}}^{\mathfrak{P}}}
\def\id{\hbox{1\hskip -3pt {\sf I}}}

\newcommand{\oP}{\overline{\mathcal{P}}_\alpha}

\title{Categories of generalized thread quivers}
\date{\today}

\author{Charles Paquette}\address{Department of Mathematics and Computer Science, Royal Military College of Canada,
Kingston, ON K7K 7B4, Canada}
\email{charles.paquette.math@gmail.com}

\author{Job Daisie Rock}
\address{Department of Mathematics W16, Ghent University, 9000 Ghent, East Flanders, Belgium}
\email{job.rock@ugent.be}

\author{Emine Y{\i}ld{\i}r{\i}m}
\address{School of Mathematics, University of Leeds, Leeds, LS2 9JT, UK}
\email{emine.y.yildirim@gmail.com}

\begin{document}

\begin{abstract}
    We study the representation category of thread quivers and their quotients. A thread quiver is a quiver in which some arrows have been replaced by totally ordered sets. Pointwise finite-dimensional (pwf) representations of such a thread quiver admit a Krull--Remak--Schmidt--Azumaya decomposition.
    We show that an indecomposable representation is induced from an indecomposable representation of a quiver obtained from the original quiver by replacing some of its arrows by a finite linear $\mathbb{A}_n$ quiver.
    We study injective and projective pwf indecomposable representations and we fully classify them when the quiver satisfies a mild condition. 
    We give a characterization of the indecomposable pwf representations of certain categories whose representation theory has similar properties to finite type or tame type. We further construct new hereditary abelian categories, including a Serre subcategory of pwf representations that includes every indecomposable representation.
\end{abstract}

\maketitle

\setcounter{tocdepth}{1}
\tableofcontents

\section*{Introduction}

\subsection*{History}

We let $k$ denote a field, which we assume is algebraically closed.
This is mainly for simplicity, but also required in some applications and examples.
In representation theory of finite-dimensional algebras, a classical problem is to understand the category of finitely generated modules over a finite-dimensional (associative) $k$-algebra $A$.
A particularly successful approach to achieve this is to use categorical methods, as was done by Gabriel and his school; see \cite{GZ67}.
One can think of $A$ as a Hom-finite $k$-linear category with finitely many objects and with split idempotents.
Such a category is also called a finite spectroid (see \cite{GR97}). In this way, the category of left $A$-modules is the category ${\rm Fun}(A, {\rm Mod}(k))$ of covariant $k$-linear functors from $A$ to the category ${\rm Mod}(k)$ of $k$-vector spaces.
As shown by Gabriel, the category $A$ is given by a finite quiver with relations: one can find a unique finite quiver $Q$ such that the path category $kQ$ admits a surjective and full functor onto $A$, and the kernel of this functor is an ideal $I$ in $kQ$ called admissible.
That is, $A$ is the quotient $kQ / I$.

Categories given by infinite quivers also play an important role.
For instance, there has been a lot of work on studying an algebra using covering techniques, especially in the classification of self-injective algebras, the study of representation-finite and minimal representation-infinite algebras, and in the study of the Brauer-Thrall conjectures; see \cite{Bon} for a review of some of the techniques.
Given an algebra $A$, we can usually find its universal cover.
This is a functor $F:B \to A$, where $B$ is some Hom-finite $k$-linear category that is given by a (generally infinite) quiver with relations.
Again, this means that there exists a (possibly infinite) quiver $\Gamma$ such that $B$ is the quotient of the path category $k\Gamma$ by some nice ideal. 

Coming from a different direction, Berg and van Roosmalen introduced the concept of a \emph{thread quiver}, where some arrows of the quiver are replaced by locally discrete totally ordered sets \cite{BvR14}. Every element of such a totally ordered set has an immediate predecessor (except the minimum) and immediate successor (except the maximum).
They studied Auslander-Reiten theory in this setting.

For general spectroids (Hom-finite $k$-linear categories with splitting idempotents), it is not true that such a category is controlled by a quiver.
However, there are some cases where one could still use a quiver to understand a spectroid, even if the spectroid is not a quotient of a path category.
An extreme example where quivers seem irrelevant is the totally ordered set $\mathbb{R}$, which can be seen as a $k$-linear category in a natural way: the object set is the set of real numbers and there is a non-zero $\Hom$ space $\Hom(x,y)=k$ if and only if $x\leq y$.

Examples such as the one given above have gained popularity in the recent years. In persistence theory, especially in one-parameter persistence theory, one can study some features of data sets by associating to them a \emph{barcode}, which is a decomposition of a representation over the real line $\mathbb{R}$ into indecomposables \cite{CB15}. The knowledge of this decomposition allows one to extract some non-trivial information on the shape of the data set \cite{CSGO16}.
One could also study the representations of a closed real interval and think of this as an $\mathbb{A}_2$ quiver where the arrow is ``replaced" by the corresponding open real interval.

Replacing the arrow in an $\mathbb{A}_2$ quiver with a real interval does not fit in Berg and van Roosmalen's framework, since a real interval is not a locally discrete totally ordered set.
In the present paper, we start by taking an arbitrary quivers $Q$ and replacing each arrow with \emph{any} totally-ordered set.
One may consider the empty set, finite sets, locally discrete sets, the real numbers, or even sets with higher cardinalities.

 Multiparameter persistence theory from representation theoretic perspectives has gained popularity in Topological Data Analysis (TDA); see for instance \cite{BBOS20, BBH24, ABH24}. In $n$-parameter persistence theory, one is interested in pointwise finite-dimensional representations of posets, often $\mathbb{R}^n$.  The discrete analogue of $\mathbb{R}^n$ is the representation of the lattice $\mathbb{Z}^n$, which we can think of as the infinite $n$-dimensional grid quiver (with the proper commutativity relations, making all parallel paths equal). Thread quivers with relations allow us to get interesting subcategories of $\mathbb{R}^n$ that contain $\mathbb{Z}^n$, and with some continuous features. For instance, if we take the full subcategory of $\mathbb{R}^2$ consisting of all vertical and horizontal lines containing lattice points, then we end up with a continuous threading of the infinite grid quiver, with an admissible ideal of relations. Many locally discrete posets, like $\mathbb{Z}^n$, can be represented by thread quivers. The theory that we develop in this paper allows us to understand the representation theory of various such posets.

\subsection*{Outline and contributions}
Section~\ref{sec:threads} gives an introduction to our new definition of thread quivers. A thread quiver $(Q,\mathcal{P})$ is a quiver $Q$ paired with a collection of totally ordered sets $\mathcal{P}_\alpha$, one for each arrow $\alpha$ in $Q$.
We define the path category $\C$ of a thread quiver, by ``replacing'' each arrow $\alpha$ with the set $\mathcal{P}_\alpha$ (Section~\ref{sec:path category}), and consider the pointwise finite-dimensional representations.
We define noise representations, which are entirely supported on $\C$ within one arrow $\alpha$ of $Q$, and noise free representations, which contain no noise direct summands (Definition~\ref{def:noise-and-noise-free}).

To be more general, we then introduce (weakly) admissible ideals $\I$ and consider quotients $\Cbar = \C/\I$ in Section~\ref{sec:quotients}.
We consider some partitions of the objects of $\Cbar$ into intervals. For each such partition, there exists a corresponding subthreading obtained by selecting the vertices of $Q$ and, for each arrow $\alpha$ of $Q$, some elements of $\mathcal{P}_\alpha$ determined by the partition.
We say that the subthreading is finite when we choose finitely many points from each $\mathcal{P}_\alpha$.
We then study the notion of completion $\overline{\I}$ of the ideal $\I$ with respect to a partition.

In Section~\ref{sec:functors} we use the partitions and completions to study certain Verdier localizations of the corresponding category $\overline{\mathcal{A}}  = \mathcal{A}/\overline{\I}$.
We conclude Section~\ref{sec:functors} by showing that the subcategory corresponding to the choice of subthreading is equivalent to the induced localization (Corollary~\ref{cor:equivalence from localized to sample on reps}).
This allows us to define an induction functor from representations of a finite subthreading of $\overline{\mathcal{A}}$ to representations of $\overline{\mathcal{A}}$.
We use these induction functors along with the usual restriction functors in Section~\ref{sec:2ndDecom} in order to study representations of $\Cbar$.

In Section~\ref{sec:2ndDecom}, we explain how to decompose any noise free representation into indecomposable representations. Such an indecomposable representation is obtained (induced) from an indecomposable representation of a finite subthreading of $Q$. Our first decomposition theorem, combined with this, yields the following. 

\begin{thm}[Theorems~\ref{thm:noise and noise free decomposition} and~\ref{thm:second decomposition theorem}]
    Let $M$ be a pointwise finite-dimensional representation of $\Cbar$. Then we have the following.
    \begin{enumerate}
        \item $M$ decomposes uniquely into indecomposable representations.
        \item Each indecomposable summand of $M$ is either noise or it is noise free and induced from an indecomposable representation of a finite subthreading of $\Cbar$.
        \item If $Q$ is finite, then there are only finitely many indecomposable summands of $M$ that are noise free.
    \end{enumerate}
\end{thm}

In Section~\ref{sec:proj-inj}, we completely classify the projective and injective pointwise finite-dimensional representations when $\Cbar$ is left and right $Q$-bounded. We say that $\Cbar$ is left (or right, respectively) $Q$-bounded when there are only finitely many paths of $Q$ ending (starting, respectively) at a given vertex of $Q$, that remain non-zero in $\Cbar$ (Definition~\ref{def:Q bounded}). Given an interval $J$ over some $\mathcal{P}_\alpha$, we define a pointwise finite-dimensional representation $P_J$ ($I_J$, respectively) which can be thought of as the direct limit (inverse limit, respectively) of the representable projective representations $\Hom_\Cbar(j,-)$ (corepresentable injective representations $\Hom_\Cbar(-,j)$, respectively) for $j \in J$.

\begin{thm} [Theorem ~\ref{ThmAllProj}]
    If $\Cbar$ is left $Q$-bounded, then all indecomposable pointwise finite-dimensional projectives are of the form $P_J$ for an interval $J$ in some $\mathcal{P}_{\alpha}$.
    
    Dually, if $\Cbar$ is right $Q$-bounded, then every indecomposable pointwise finite-dimensional injectives is of the form $I_J$ for an interval $J$ in some $\mathcal{P}_{\alpha}$.
\end{thm}

Section~\ref{sec:exs} is devoted to illustrate many examples and applications of our results in the earlier sections.
We define new representation types for thread quivers that generalize the classical representation finite and tame types.
The first is based on the partitions mentioned earlier (Definition~\ref{def:virtually finite and tame}) and the second on dimension vectors (Definition~\ref{def:essentially finite and tame}). We call them virtually finite (or virtually tame) and essentially finite (or essentially tame), respectively. We prove how these representation types are related in Proposition~\ref{prop:essentially is virtually}.
We show that special biserial categories from thread quivers are of virtually tame type and that certain (natural) thread quivers of extended Dynkin types $\widetilde{\mathbb{A}}$ and $\widetilde{\mathbb{D}}$ are virtually tame and essentially tame, respectively.

Finally, in Section~\ref{sec:hereditary}, we show that the category of pointwise finite-dimensional representations is hereditary when $\Cbar$ is left or right $Q$-bounded.
We consider the category of quasi noise free representations, which contains all indecomposable pointwise finite-dimensional representations.
A quasi noise free representation has at most finitely many noise direct summands supported on each arrow.
We show that this category is always abelian and hereditary.
We also show that, when $\C$ is left $Q$-bounded, the category of finitely presented representations is abelian and hereditary.

\begin{thm}[Theorem~\ref{thm:quasi noise free is hereditary} and Propositions~\ref{prop:no infinite paths and interval finite implies hereditary}~and~\ref{prop:fp_hereditary}]
    Let $\C$ be the path category of a thread quiver $(Q,\mathcal{P})$.
    The following statements about the hereditary property hold.
    \begin{enumerate}
        \item The category $\rqnf\C$ is hereditary.
        \item If $\C$ is left $Q$-bounded, then both $\rpwf\C$ and $\rfp\C$ are hereditary.
        \item If $\C$ is right $Q$-bounded then $\rpwf\C$ is hereditary.
    \end{enumerate}
\end{thm}

\noindent We conjecture that $\rpwf\C$ is always hereditary (Conjecture~\ref{con:pwf hereditary}).

\subsection*{Future work} In a forthcoming work, we apply the representation theory of thread quivers to create new cluster categories and study them.

\subsection*{Preliminaries and conventions}

 We assume $k$ is an algebraically closed field.
For a category $\mathcal{A}$ we denote by $\mathcal{A}_0$ the set or class of objects of $\mathcal{A}$.
Given a $k$-linear category $\mathcal{A}$, we let $\rpwf(\mathcal{A})$ denote the category of  covariant $k$-linear functors from $\mathcal{A}$ to the category ${\rm mod} (k)$ of finite-dimensional $k$-vector spaces. Such a functor (also called representation, or module) will be said to be a \emph{pointwise finite-dimensional} over $\mathcal{A}$.  Note that $\rpwf(\mathcal{A})$ is $k$-linear and abelian. Unless otherwise stated, all representations are pointwise finite.

\subsection*{Acknowledgements}
CP was supported by the National Sciences and Engineering Research Council of Canada, and by the Canadian Defence Academy Research Programme.
JDR is supported by BOF grant 01P12621 from Ghent University.
EY is supported in part by the Royal Society Wolfson Award RSWF/R1/180004.

\section{Representations of thread quivers}~\label{sec:threads}

The notion of thread quiver has been introduced by Berg and Van Roosmalen in \cite{BvR14}. The idea is to consider a quiver and append to each arrow a linearly ordered set. 
Given our field $k$, we can then define the path category of this thread quiver, similar to the path category of a quiver. In \cite{BvR14}, the authors were mainly interested in the case where all linearly ordered sets are locally discrete, because that is exactly the condition needed for the category of finitely presented representations of the path category of the thread quiver to have almost split sequences.  In this paper, we do not put any restriction on the linearly ordered sets, and also consider (weakly) admissible quotients of path categories of thread quivers which has not been considered in~\cite{BvR14}.

\subsection{Thread quivers}

We let $Q = (Q_0, Q_1, s, t)$ denote any quiver, possibly infinite. The maps $s,t: Q_1 \to Q_0$ are the source and target maps, respectively, from the set of arrows $Q_1$ to the set of vertices $Q_0$. For each arrow $\alpha$ in $Q_1$, we associate a linearly ordered set (possibly empty), denoted $\mathcal{P}_\alpha$. Additionally, we let $\oP$ be the linearly ordered set obtained from $\mathcal{P}_\alpha$ by adding a minimal element $\min(\alpha)$ and adding a maximal element $\max(\alpha)$. In other words, the collection $\mathcal{P}:=\{\mathcal{P}_\alpha\}_{\alpha \in Q_1}$ gives us a function from $Q_1$ to the class of linearly ordered sets which associates to every $\alpha$ a linearly ordered set $\mathcal{P}_\alpha$. Note that $\oP$ always differs from $\mathcal{P}_\alpha$ by exactly two elements.


\begin{definition}
A \emph{thread quiver} $(Q,\mathcal{P})$ is the quiver $Q$ with the collection of linearly oriented posets $\mathcal{P}$.
\end{definition}

Let $Q$ be a quiver and $(Q,\mathcal{P})$ a thread quiver.
Notice that, for any arrow $\alpha$ in $Q$, the posets $\mathcal{P}_\alpha$ and $\oP$ are both totally ordered.
This does not change even if $\alpha$ is a loop.
We can think of the thread quiver $(Q,\mathcal{P})$ as being a collection of disjoint totally ordered sets $\{\oP\}_{\alpha\in Q_1}$. However, because of a later construction, it will be convenient to draw a thread quiver glued together using the quiver $Q$. For example, in our drawing, if $\alpha, \beta$ are arrows with $t(\alpha) = s(\beta)$, then we identify $\max(\alpha)$ with $\min(\beta)$ and label it by $s(\beta) = t(\alpha)$.

Different threadings may yield the same diagram.
For any $\alpha\in Q_1$ such that $\mathcal{P}_\alpha$ is non-empty, one may pick an arbitrary element $x$ of $\mathcal{P}_\alpha$.
We can create a new quiver $Q'$ whose vertices are $Q_0\cup\{x\}$ and whose arrows are $(Q_1\setminus\{\alpha\})\cup\{\alpha_1,\alpha_2\}$, where $\alpha_1:s(\alpha)\to x$ and $\alpha_2:x\to t(\alpha)$.
Set
\begin{align*}
    \mathcal{P}'_\beta &= \mathcal{P}_\beta \qquad \beta\neq \alpha \\
    \mathcal{P}_{\alpha_1} &= \{y\in \mathcal{P}_\alpha \mid y < x\} \\
    \mathcal{P}_{\alpha_2} &= \{y\in \mathcal{P}_\alpha \mid y > x\}.
\end{align*}
Let $\mathcal{P}_{\alpha_1}$ and $\mathcal{P}_{\alpha_2}$ inherit their order as subsets of $\mathcal{P}_\alpha$.
Now we have a new thread quiver $(Q',\mathcal{P}')$ that is essentially the same as $(Q,\mathcal{P})$.
In particular, if $\alpha$ is a loop and $\mathcal{P}_\alpha$ is nonempty, we may reinterpret $(Q,\mathcal{P})$ as having a two cycle instead of a loop.

In our drawings, we systematically put $\mathcal{P}_\alpha$ on top of the arrow $\alpha$ unless $\mathcal{P}_{\alpha}$ is empty, in which case we keep the original label of the arrow. 

\begin{example}~\label{ex:A2}  Let $Q=\mathbb{A}_2: \xymatrix{\bullet_a \ar[r]^{\alpha} & \bullet_b}$.

\begin{enumerate}
\item Consider a poset $\mathcal{P}_{\alpha}=(0,1).$ Then, $\oP = \{\min(\alpha),\max(\alpha)\}\cup (0,1)$ with $\min(\alpha) < t<\max(\alpha)$ for all $t \in (0,1)$ which yields

\[(Q,\mathcal{P})= \xymatrix{ {}_{a=\min(\alpha)}\bullet \ar@{=>}^{\mathcal{P}_{\alpha}}[r] & \bullet_{b=\max(\alpha)} }\]

\item Let us now consider a poset $\mathcal{P}_{\alpha}=[0,1].$ Then, $\oP = \{\min(\alpha),\max(\alpha)\}\cup [0,1]$ with $\min(\alpha)< 0$ and $1<\max(\alpha)$ which yields

\[(Q,\mathcal{P})= \xymatrix{  \bullet_{a} \ar[r] & \bullet_0 \ar@{=>}^{\mathcal{P}_{\alpha}}[r] & \bullet_1 \ar[r] & \bullet_{b} }\]

Note that when we consider $\mathcal{P}_{\alpha}=[0,1]$, the minimum and the maximum of $\oP$ are labeled by $a$ and $b$, respectively. Notice that for this example, we could start with an $A_4$ quiver and thread its middle arrow with $(0,1)$ and the others with the empty set to obtain an equivalent diagram.
\end{enumerate}
\end{example}

\begin{example}~\label{running-ex}
    Let us give another interesting example. Consider the quiver $Q$ of type $\mathbb{D}_4$ as follows.
\[ \xymatrix@R=1ex{               & & \bullet_c \\ 
                  \bullet_a\ar^{\alpha}[r] & \bullet_b\ar^{\beta}[ur] \ar_{\gamma}[dr] & \\
                                & & \bullet_d}\]

Now, we choose $\mathcal{P}_{\alpha}=(0,4), \mathcal{P}_{\beta} = \mathbb{R}$ and $\mathcal{P}_{\gamma}=\emptyset.$ We obtain the thread quiver

\[\xymatrix@R=2ex{               & & \bullet_c \\ 
                  \bullet_a\ar@{=>}^{(0,4)}[r] & \bullet_b\ar@{=>}^{\mathbb{R}}[ur] \ar_{\gamma}[dr] & \\
                                & & \bullet_d}\]

\end{example}

\subsection{The path category of a thread quiver.}\label{sec:path category}

In this section, we will construct a path category associated to any thread quiver $(Q, \mathcal{P})$. For this, we first need to introduce elements for any pair $(x,y)$ with $x \leq y$ in a given $\oP$.
These elements will generate the morphisms in the path category.

For each $x,y \in \oP$ with $x \le y$, we associate a symbol $\eta_{y,x}$. Note that if $\alpha$ is not a loop, we identify $\eta_{\max(\alpha),\min(\alpha)}$ with $\alpha$; when $\alpha: x \to x$ is a loop, the elements $\eta_{x,x}$ and $\alpha$ are not identified.
We simply denote $\eta_{\max(\alpha),x}$ by $\eta_{ t(\alpha), x}$ and $\eta_{x, \min(\alpha)}$ by $\eta_{x, s(\alpha)}$. 

Let $k$ be any field. We construct a $k$-linear category $\C$ of $(Q,\mathcal{P})$, which we call the \emph{path category} of $\C$, as follows.

Consider the surjective map
$$\varphi:  \bigsqcup_{\alpha\in Q_1} \oP\longrightarrow Q_0 \cup \left( \bigcup_{\alpha\in Q_1} \mathcal{P}_\alpha\right)$$
that is the identity on $ \bigsqcup_{\alpha\in Q_1} \mathcal{P}_\alpha$ and such that for an arrow $\alpha$, we have $\varphi(\min(\alpha)) = s(\alpha)$ and $\varphi(\max(\alpha)) = t(\alpha)$.
For $x,y$ in the domain of this map, we write $x \sim y$ if $\varphi(x)=\varphi(y)$. The object set of $\C$ is given by the equivalence classes
$$\bigsqcup_{\alpha\in Q_1} \oP\Bigg/\sim$$

Hence, the set $$Q_0 \cup \left( \bigcup_{\alpha\in Q_1} \mathcal{P}_\alpha\right)$$
is a complete set of representatives of these equivalence classes, and we use these representatives to denote our objects.

As for the Hom spaces between two objects $x,y$, we have that $\Hom_{\C}(x,y)$ is a $k$-vector space with basis given as follows. 
\begin{itemize}
    \item Case 1: $x\leq y\in \oP$ for some $\alpha\in Q_1$.
    The basis is given by
    \begin{displaymath}
     \{\eta_{y,x}\} \cup \{\eta_{y,s(\alpha)} \,p\, \eta_{t(\alpha),x} \mid p:t(\alpha)\to s(\alpha) \text{ is a path in } Q\} .
    \end{displaymath}
    \item Case 2: $x\in \oP$ and $y\in\Pbar_\beta$, where if $\alpha=\beta$ then $x>y$.
    The basis is given by
    \begin{displaymath}
     \{\eta_{y,s(\alpha)} \,p\, \eta_{t(\alpha),x} \mid p:t(\alpha)\to s(\beta) \text{ is a path in } Q\} .
    \end{displaymath}
\end{itemize}

\begin{figure}[h]
    \centering
    \includegraphics[width=7.2cm]{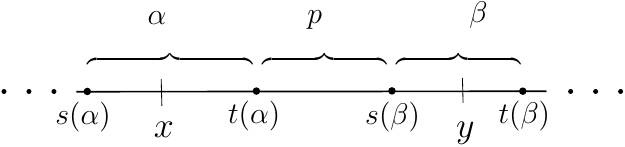}
    \caption{An illustration of $\alpha, \beta$ and $p.$}
    \label{composition}
\end{figure}

Composition of basis elements in $\C$ is induced from path concatenation in $Q$ together with the relations $\eta_{z,y}\eta_{y, x}=\eta_{z,x}$ whenever $x \le y \le z$ in a given $\oP$, for a non-loop $\alpha$.
When $\alpha$ is a loop, we use the same relations \emph{except} with the difference that $\eta_{t(\alpha),x}\eta_{x,s(\alpha)}=\alpha$.
We extend this composition rule bilinearly.

Given any $x \in \oP$, we identify $\eta_{x,x}$ with the identity morphism at $x$.
We denote it by $e_x$ and call it the \emph{trivial path} at $x$.
The elements of the form $\eta_{y,x}$ and $\eta_{y,s(\alpha)} \,p\, \eta_{t(\alpha),x}$ as above are called \emph{path-like} elements.
Hence, every $\Hom$ space has a basis given by path-like elements.
In this way, we have defined a $k$-linear skeletal category $\C$. We note that categories of this form include path categories of any quiver by taking all $\mathcal{P}_\alpha$ to be empty. 

\begin{example}
\begin{enumerate}
    \item Recall $\mathcal{P}_{\alpha}=(0,4), \mathcal{P}_{\beta} = \mathbb{R}$ in Example~\ref{running-ex}. Take an element $x\in \mathcal{P}_{\alpha}$ and an element $y\in \mathcal{P}_{\beta}$, then the $\Hom$ space between the objects $x, y$ in $\C$ is
    \[
        \Hom(x,y)= \Span_k\{\eta_{yb}\eta_{bx}\} \cong k.
    \]
    \item Consider the quiver $Q$ as on the left and the thread quiver $(Q,\mathcal{P})$ of $Q$ on the right as shown below.
     \[
        \xymatrix{\bullet_a\ar@/^2ex/^{\alpha}[r] & \bullet_b\ar@/^2ex/^{\beta}[l]}
        \qquad \qquad
        \xymatrix{\bullet_a\ar@/^2ex/@{=>}^{(0,4)}[r] & \bullet_b\ar@/^2ex/@{=>}^{\mathbb{R}}[l]}
    \]
    Similarly, take $x\in (0,4)$ and $y\in \mathbb{R}$.
    Now, the $\Hom$-space is infinite-dimensional and spanned as follows.
    \[
        \Hom(x,y)= \Span_k\{\eta_{yb}\eta_{bx},\cdots, \eta_{yb}(\alpha\beta)^i\eta_{bx},\cdots\} \cong \bigoplus_{i=1}^\infty k
    \]
\end{enumerate}
    
\end{example}

If $Q$ is a tree quiver then its path algebra $kQ$ is equivalent to an incidence algebra (possibly without unit).
In this case, the path category $\C$ of $(Q,\mathcal{P})$ is equivalent to an incidence category.

\begin{definition}~\label{interval-fin} A quiver $Q$ is \emph{interval finite} if $Q$ admits at most finitely many paths from $x$ to $y$, for any pair $(x,y)$ of vertices in $Q_0$.
\end{definition}

It is not hard to see that the category $\C$ is $\Hom$-finite precisely when $Q$ is interval finite.

\subsection{Pointwise finite-dimensional interval representations}

Recall that a \emph{representation} is a covariant $k$-linear functor $M: \C \to \Mod\ k$. A \emph{pointwise finite-dimensional} representation is a covariant $k$-linear functor $M: \C \to \mmod k$.  We respectively denote by $\Rep\ \C$ and $\rpwf \C$ the categories of all representations and all pointwise finite-dimensional representations. These categories are $k$-linear and abelian. In order to define our first family of representations, we need the following definition.

\begin{definition}\label{def:interval}
    Let $P$ be any linearly ordered set. A subset $I$ of $P$ is called \emph{interval} if it satisfies the condition that if $x,y \in I$ and $x<z<y$ in $P$ then $z\in I$. Moreover, all intervals considered in this paper are assumed to be non-empty.
\end{definition}
In the present paper, we will not be considering intervals starting and ending in distinct arrows of the quiver.

\begin{example}
    In Figure~\ref{int} on the left, if we consider all elements between points $x$ and $y$ from $(0,4)$, this yields an example of an interval as in Definition~\ref{def:interval}. But if we consider $x$ and $y$ as shown in the second picture in Figure~\ref{int}, this is not an example of the intervals we consider.
\end{example}

\begin{figure}[h]
    \centering
    \includegraphics[width=6.7cm]{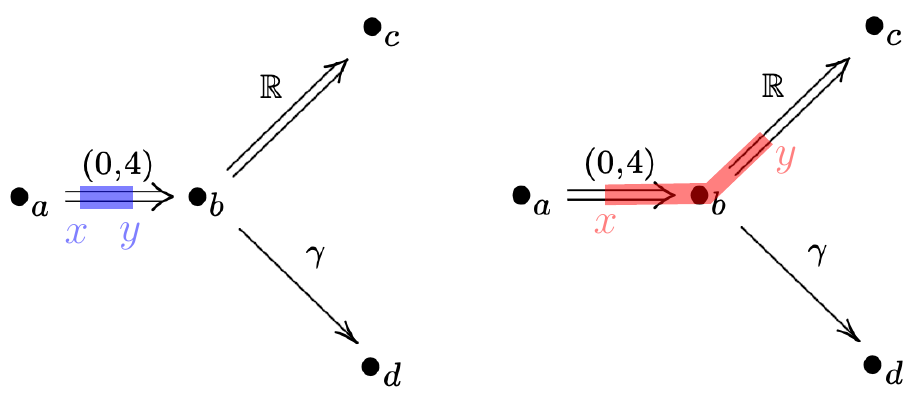}
    \caption{The blue region on the left corresponds to the interval $(x,y)\subset (0,4)=\mathcal{P}_\alpha$. For our consideration, the red region on the right does not correspond to an interval.}
    \label{int}
\end{figure}

Next, we define interval representations.

\begin{definition}
  Let $I$ be any interval in $\oP$ for some $\alpha \in Q_1$, excluding the case where $I=\oP$ and $\alpha$ is a loop. We define an \emph{interval representation} $M_I$ to be the pointwise finite-dimensional representation of $\C$ such that $M_I(x) = k$ if $x \in I$ and $M_I(x)=0$ otherwise. Moreover, if $x \le y$ in $I$, then the corresponding map $M(\eta_{y,x}): M(x) \to M(y)$ is the identity map.   For all other path-like elements $\rho$, $M(\rho)=0$.
\end{definition}

We notice that any such interval representation has a one-dimensional endomorphism algebra. In particular, these representations are all indecomposable. Now we recall the decomposition theorem for a linearly ordered set.  It can be found in \cite[Section 3.6, page 30]{GR97} without a proof, but Crawley-Boevey proved it in \cite{CB15}. In the case where the poset is that of real numbers $\mathbb{R}$, another proof can be found in \cite[Lemma 2.4.14]{IRT22}.

\begin{theorem}[\cite{GR97, CB15, IRT22}]\label{thm:c-b}
Let $Q$ be the $\mathbb{A}_2$ quiver with arrow $\alpha$ and let $\mathcal{P}_\alpha$ be a linearly ordered set, and consider the corresponding path category $\C$. Then any representation $M \in \rpwf\C$ decomposes uniquely as a direct sum of interval representations.
\end{theorem}

We also note that the uniqueness in the theorem comes from the celebrated Krull--Remak--Schmidt--Azumaya theorem, since any interval representation has local endomorphism algebra.

\subsection{First decomposition theorem}

In this subsection, $\C$ is the path category of a given thread quiver $(Q, \mathcal{P})$.
Recall that any pointwise finite-dimensional representation decomposes uniquely as a direct sum of indecomposable representations, all having local endomorphism algebras. It is found in \cite[Page 30]{GR97} but for a complete proof see  \cite{BCB20}. Recall that the \emph{support}  of a representation 
$M$, denoted $\supp(M)$, is the set of objects $x$ of $\C$ with $M(x) \ne 0$.

 The next Definition~\ref{def:confinement} and Lemma~\ref{lem:noise summand} are generalized from \cite[Definition 2.14, Lemma 2.15]{HR22}.
\begin{definition}\label{def:confinement}
      Let $R$ be a subset of objects of $\C$ included in a given $\oP$ for some $\alpha \in Q_1$.
    Let $M$ be a pointwise finite-dimensional representation of $\C$.
   
    The \emph{confinement of $M$ to $R$}, denoted $M|_R$, is the representation of $\C$ given by the following. For each object $x$ of $\C$ and each path-like element $\eta$ in $\C$:
    \begin{align*}
        M|_R(x) &:= \begin{cases}
            M(x) & x\in R \\
            0 & x\notin R
        \end{cases} \\
        M|_R(\eta) &:= \begin{cases}
            M(\eta) & \text{if for all path-like elements } \beta,\gamma \text{ such that }\eta=\gamma e_y\beta,\ \, y \in R \\
            0 & \text{otherwise}.
        \end{cases}
    \end{align*}
    
\end{definition}

\begin{lemma}\label{lem:noise summand}
    Let $M$ a pointwise finite-dimensional representation of $\C$.
    Let $R=\oP$, for some arrow $\alpha$ in $Q$.
    If $U$ is a direct summand of $M|_R$ and $\supp(U)\cap Q_0=\emptyset$ then $U$ is a direct summand of $M$.
\end{lemma}
\begin{proof}
    This proof follows from the same kind of arguments used in the proof of \cite[Lemma 2.1]{HR22}.
\end{proof}

Representations such as the summand $U$ from Lemma~\ref{lem:noise summand} will be called \emph{noise} representations, as in the next definition.

\begin{definition}~\label{def:noise-and-noise-free}
     Let $M$ be a pointwise finite-dimensional representation of $\C$. We call $M$  
     \begin{enumerate}[$(1)$]
         \item \emph{noise} if the support of $M$ is such that $\supp(M) \cap Q_0 = \emptyset$;
         \item \emph{noise free} if no direct summand of $M$ is noise.
     \end{enumerate}
\end{definition}

\begin{remark}
  Equivalently, by Lemma~\ref{lem:noise summand}, we can say that $M$ is noise free if, for each arrow $\alpha$ in $Q$, every indecomposable summand of $M|_{\oP}$ has support in $Q_0$.
\end{remark}

\begin{example}
    \begin{enumerate}
        \item Let $x$ be an object in $\C$ but not in $Q_0$. Then the simple representation at $x$ is noise. However, if $x$ is an object in $\C$ such that $x\in Q_0$, then the simple representation at $x$ is noise free. 
        \item Suppose an arrow $\alpha$ is threaded with a copy of $\mathbb{R}$. Then any representation of $\mathbb{R}$ can be thought of as a noise representation of $\C$ where the support is contained in $\mathcal{P}_\alpha=\mathbb{R}$.
        \item Let $M$ be noise free and $N$ be noise. Then $M\oplus N$ is neither noise nor noise free.
    \end{enumerate}
\end{example}

\begin{theorem}\label{thm:noise and noise free decomposition}
    Let $\C$ is the path category of a thread quiver $(Q, \mathcal{P})$, and $M$ be a pointwise finite-dimensional representation of $\C$.
    Then $M$ decomposes uniquely (up to isomorphism of summands) into $M_{NF} \oplus M_N$ where $M_{NF}$ is noise free and $M_N$ is noise.
    Moreover, $M_N = \bigoplus_{\alpha\in Q_1} U_{\alpha}$, where $\supp(U_{\alpha})\subseteq \mathcal{P}_\alpha$.
\end{theorem}

\begin{proof}
    For each arrow $\alpha$ in $Q$, using Remark~\ref{lem:noise summand}, denote by $U_\alpha$ the direct sum of noise summands of $M|_{\oP}$.
    By Lemma~\ref{lem:noise summand}, each $U_\alpha$ is a summand of $M$.
    Then $M=M_{NF}\oplus (\bigoplus_{\alpha\in Q_1} U_\alpha)$.
    By construction, $M_{NF}$ is noise free.
    Let $M_N = \bigoplus_{\alpha\in Q_1} U_\alpha$.
    By Theorem~\ref{thm:c-b}, this decomposition is unique up to isomorphism of summands. This concludes the proof of the theorem.
\end{proof}

\begin{definition}\label{def:quasi noise free}
    A representation $M \in \rpwf \C$ is \emph{quasi noise free} if for every arrow $\alpha \in Q_1$, the representation $U_\alpha$ in Theorem~\ref{thm:noise and noise free decomposition} has finitely-many direct summands.
\end{definition}

The motivation behind this definition is that if $M$ is quasi noise free, then the noise representation supported on any arrow is a finite direct sum of interval representations. 

We let $\rqnf\C$ denote the full subcategory of $\rpwf\C$ consisting of the quasi noise free representations.

\section{Ideals, quotients, partitions and completions}\label{sec:quotients}

In this section, we consider some ideals of $\C$ and their corresponding quotients. Motivated by the classical notion of admissible ideals for path categories of quivers, we define (weak) admissibility for ideals in $\C$. These notions will be important in Section \ref{sec:proj-inj}.

Given a quotient category $\Cbar$ of $\C$ by an arbitrary ideal $\I$, we consider some partition $\mathfrak{P}$ of the objects of $\Cbar$ (equivalently, the objects of $\C$) into intervals, which we call valid partitions. Given a valid partition, we can select one representative from each cell and then get a set $\mathcal{M}$ of objects of $\Cbar$ that we call sample points. The full subcategory of $\Cbar$ containing the objects $\mathcal{M}$ is denoted $\Cbar(\mathcal{M})$. We will see that the inclusion functor $\iota: \Cbar(\mathcal{M}) \to \Cbar$ admits a left inverse when $\I$ satisfies an additional property that we refer to as being $\mathfrak{P}$-complete. That same property also allows us, in the next section, to localize $\Cbar$ at a set of morphisms corresponding to the partition $\mathfrak{P}$.

\subsection{Weakly admissible ideals}\label{sec:ideals}

We want to make our categories more general, and consider some well-behaved quotient categories of a path category of a thread quiver. For path categories of a quiver, an admissible ideal usually requires generators to be linear combinations of paths of length at least two, and some additional conditions so that every oriented cycle is nilpotent. With the first condition, the quiver represents the objects and irreducible morphisms between them. When working with non-discrete posets, it is convenient to have several variations the first condition.

\begin{remark} Let $x,y \in \C$. A morphism $f: x \to y$ is irreducible if for any factorization $f = f_2f_1$, $f_1$ or $f_2$ is an isomorphism. In our category, irreducible morphisms are the non-zero scalar multiples of all $\eta_{y, x}$ where there is no $z$ with $x < z < y$.
\end{remark}

\begin{definition}
    Given a thread quiver $(Q, \mathcal{P})$ and its corresponding path category $\mathcal{C}$, we denote by $\mathcal{J}$ the ideal generated by all path-like elements that are different from the $e_x = \eta_{x,x}$. We call this the \emph{path-like radical} of $\mathcal{C}$.
\end{definition}

\begin{definition} \label{weakly admissible}
    Let $\I$ be an ideal of $\C$ and $\Cbar=\C/\I$.  We say $\I$ \emph{weakly admissible} when the following are satisfied: 
\begin{enumerate}[(i)]
    \item all algebras $\End_{\Cbar}(x)$, for $x \in \C$, are local;
    \item $\Hom_{\Cbar}(x,y)$ is finite-dimensional for all $x,y \in \C$; and
    \item $\I$ is generated by elements in $\mathcal{J}$.
\end{enumerate}  
 If, additionally, $\I$ is generated by elements in $\mathcal{J}^2$, then $\I$ is called \emph{admissible}.
\end{definition}


The weakly admissible condition implies that the quotient category is Hom-finite with splitting idempotents. The ``weakly'' in the above definition comes from the fact that we may, {\em a priori}, include an irreducible morphism of $\C$ in $\I$, which would not happen for admissible ideals. Note that in \cite[Page 81]{GR97}, the notion of admissibility for the path category $kQ$ of a quiver is weaker than our notion of admissibility in that case.
Note also that all admissible ideals in the sense of \cite{JR22} satisfy (i), (ii), and a stronger condition than being generated by elements of $\mathcal{J}^2$. They are thus admissible ideals as defined in the present paper.

The quotient category of $\C$ by $\I$ will be denoted by $\Cbar$. The category of all representations of $\Cbar$ (that is, $k$-linear covariant functors $\Cbar \to \Mod\ k$) is denoted by $\Rep\ \Cbar$. The subcategory of all pointwise finite-dimensional representations is denoted $\rpwf\Cbar$. We note that $\rpwf\Cbar$ is the full subcategory of $\rpwf\C$ such that for every object $M$ in $\rpwf\Cbar$ and every $f\in \I$, we have $M(f)=0$.

\begin{remark}~\label{rem:Cbar}
    Our first decomposition theorem, Theorem~\ref{thm:noise and noise free decomposition}, applies to representations in $\rpwf\Cbar$ because $\rpwf\Cbar$ is a full subcategory of $\rpwf\C$ that is closed under direct sums and direct summands.
\end{remark}

\begin{example}~\label{ex:with ideal}
\begin{enumerate}
    \item  Consider one of our previous Example~\ref{ex:A2}(1). Recall that $(Q,\mathcal{P})$ is:

\[(Q,\mathcal{P})= \xymatrix{ \bullet_0 \ar@{=>}^{\mathcal{P}_{\alpha}}[r] & \bullet_1 }\]
where $\mathcal{P}_\alpha$ was one of the continuous intervals $(0,1)$ or $[0,1]$. Let us now take the following set of relations for $(Q,\mathcal{P})$.

\[\mathcal{R}=\{\eta_{y,x}:\ y-x\geq 1/2\}\] and set $\mathcal{I} = \langle \mathcal{R} \rangle.$ Observe that $\I$ is admissible.

\item Another example we can consider continues from Example~\ref{running-ex}.

\[\xymatrix@R=2ex{ & & \bullet_c \\ 
                  \bullet_a\ar@{=>}^{(0,4)}[r] & \bullet_b\ar@{=>}^{\mathbb{R}}[ur] \ar_{\gamma}[dr] & \\
                                & & \bullet_d}\] 

Now, consider the ideal $\mathcal{I}$ generated by \[\mathcal{R}=\{\eta_{y,b}\eta_{b,x}: x\in \mathcal{P}_{\alpha}=(0,4)\ \text{ and }  y\in \mathcal{P}_{\beta}=\mathbb{R}\}\] which is a weakly admissible ideal.
\end{enumerate}

\end{example}

\subsection{Partitions and completion of ideals}

Let $\C$ be the path category from a thread quiver $(Q,\mathcal{P})$.
Let $\I$ be any ideal in $\C$ and $\Cbar=\C/\I$.

\begin{definition}
    A partition of the object set of $\mathcal{C}$ is said to be \emph{valid} if, (i) for each $x\in Q_0$, the set $\{x\}$  is a cell and (ii) each poset $\mathcal{P}_\alpha$ for an arrow $\alpha$ is a union of finitely many interval cells.
\end{definition}

We can associate a valid partition to some representations. Let $M$ be a representation in $\rpwf \C$.
For each $\alpha\in Q_1$, let $\mathfrak{P}_{M,\alpha}$ be the unique partition of $\mathcal{P}_\alpha$ such that $x\leq y$ are in the same cell of $\mathfrak{P}_{M,\alpha}$ if and only if for all $x\leq z \leq y$ we have $M(\eta_{zx})$ and $M(\eta_{yz})$ are isomorphisms. 

Define \[\mathfrak{P}_M := \{\{x\}|x\in Q_0\}\cup \left( \bigcup_{\alpha\in Q_1} \mathfrak{P}_{M,\alpha} \right). \]
Notice that for representations $M$ and $N$, the partition $\mathfrak{P}_{M\oplus N}$ is given by the intersection of cells of the partitions $\mathfrak{P}_M$ and $\mathfrak{P}_N$.\label{partition of the direct sum}
In particular, $\mathfrak{P}_{M\oplus N}$ is a refinement of both $\mathfrak{P}_M$ and $\mathfrak{P}_N$.

\begin{proposition}\label{prop:noise free inducess valid partition}
    Let $M$ be a representation of $\rpwf \Cbar$.
    Then $\mathfrak{P}_M$ is a valid partition if and only if $M$ is quasi noise free.
\end{proposition}
\begin{proof}
    Assume that $M$ is quasi noise free. By construction, we see that each $x\in Q_0$ is contained in cell $\{x\}\in\mathfrak{P}_M$.
    We need to show that $\mathfrak{P}_{M,\alpha}$ is finite, for each $\alpha\in Q_1$. Recall we define $M|_\alpha$ to be the confinement $M|_{\oP}$ (see Definition~\ref{def:confinement}).

    Notice, for each $\alpha\in Q_1$, $M|_\alpha$ decomposes uniquely up to isomorphism into a direct sum of strings (and possibly bands if $\alpha$ is a loop) \cite{GR97,CB15,HR22}.
    Since $M$ is quasi noise free, there are finitely-many string (or band) summands with support at $s(\alpha)$ or $t(\alpha)$.
    Since $M$ is finite-dimensional at both $s(\alpha)$ and $t(\alpha)$ and quasi noise free, we see that $M|_\alpha$ is a finite direct sum.
    Thus, the partition $\mathfrak{P}_{M,\alpha}$ must contain finitely many cells.
    Therefore, $\mathfrak{P}_M$ is valid.

    Conversely, assume that $\mathfrak{P}_M$ is valid. Let $\alpha \in Q_1$. It suffices to prove that $M|_{\alpha}$ is a finite direct sum of indecomposable representations. Assume otherwise, that is, we have countably-many distinct summands $Z_1, Z_2, \ldots$ all supported on $\mathcal{P}_\alpha$. For each $i$, let $x_i \in \mathcal{P}_\alpha$ be in the support of $Z_i$. Since $\mathfrak{P}_M$ is valid, infinitely many of the $x_i$ lie in the same cell $X$ of $\mathfrak{P}_M$. This yields that those infinitely many summands all have $X$ in their support, which would contradict the fact that $M|_{\alpha}$ is pointwise finite-dimensional.
\end{proof}

\begin{definition}
    From a valid partition $\mathfrak{P}$, we select a set $\mathcal{M}$ of objects from $\Cbar$ in such a way that each cell of $\mathfrak{P}$ contains exactly one point from $\mathcal{M}$. Note that $Q_0 \subseteq \mathcal{M}$.
Such a set $\mathcal{M}$ is called a \emph{sample}\label{text:sample definition maybe} of $\mathfrak{P}$.
\end{definition}

\medskip

\begin{remark}\label{rmk:axiom of choice}
    Note that if $\mathfrak{P}$ is finite (so when the quiver $Q$ is finite) or if we can compute each point of the sample (e.g. computing a half-way point), then we do not need the axiom of choice to choose $\mathcal{M}$. 
    However, in general, we need the axiom of choice to choose $\mathcal{M}$.
\end{remark}

We let $\Cbar(\mathcal{M})$ denote the full subcategory of $\Cbar$ generated by the objects in $\mathcal{M}$. Note that when $\I=0$, this category is the path category of a quiver which is obtained from $Q$ by replacing each arrow by a finite linear $\mathbb{A}_n$ quiver. When $Q$ is finite and $\I$ is weakly admissible, $\Cbar(\M)$ gives rise to a finite-dimensional algebra whose Gabriel quiver is obtained from $Q$ by replacing each arrow by a finite linear $\mathbb{A}_n$ quiver.

We let $\iota: \Cbar(\mathcal{M}) \to \Cbar$ be the inclusion functor. Let $\Sigma_\mathfrak{P}$ be the set of morphisms that are non-zero scalar multiples of the $\eta_{yx}$ in $\Cbar$ where $x, y$ are in a given cell of $\mathfrak{P}$. We would like to have a left inverse $\mathbf{r}: \Cbar \to \Cbar(\mathcal{M})$ to $\iota$ that inverts the morphisms in $\Sigma_\mathfrak{P}$. We will see that the existence of such an inverse is tied to $\mathfrak{P}$-complete ideals, as we will define below.

\begin{definition}\label{def:p equivalent}
    Let $f_1:x_1\to y_1$ and $f_2:x_2\to y_2$ be morphisms in $\C$ with $x_1, x_2$ in a given cell $X$ of $\mathfrak{P}$ and $y_1, y_2$ in a given cell $Y$ of $\mathfrak{P}$. We say that $f_1$ and $f_2$ are $\mathfrak{P}$-\emph{equivalent} if  there exist $z \in X$ and $w \in Y$ such that $\eta_{wy_1}f_1\eta_{x_1 z} = \eta_{wy_2}f_2\eta_{x_2 z}.$ 
\end{definition}

Clearly, $\mathfrak{P}$-equivalence is an equivalence relation.

\begin{definition}
\begin{enumerate}
    \item[(i)] Let $\Ihat$ be the ideal in $\C$ generated by all $f$ such that there is a $g\in \I$ such that $f$ and $g$ are $\mathfrak{P}$-equivalent. We call $\Ihat$ the $\mathfrak{P}$\emph{-completion} of $\I$ and an ideal $\I$ with $\I = \Ihat$ is called $\mathfrak{P}$\emph{-complete}.
    \item[(ii)] Define $\Chat:=\C/\Ihat$. 
\end{enumerate}
\end{definition}

If $f_2 = \lambda f_1$ for some $1\neq\lambda\in k$ and nonzero $f_1$, then $f_1$ and $f_2$ are not $\mathfrak{P}$-equivalent.
Even though $f_1$ may not be $\mathfrak{P}$-equivalent to $f_2$ in this case, if one is in $\Ihat$ then so is the other.

\begin{remark}
\begin{enumerate}
    \item We observe that if $\eta_{yx} \in \I$ for $x,y$ in a given cell, then all morphisms $\eta_{st}$ for $s,t$ in that cell are in $\Ihat$, so that the cell completely vanishes in the quotient by $\Ihat$. In particular, the objects of that cell are isomorphic to the zero object in $\Chat$.
    \item The ideal $\Ihat$ is in general not weakly admissible, even if $\I$ is.
    \item Note that if $\I=0$ then $\I$ is $\mathfrak{P}$-complete for any valid partition $\mathfrak{P}$.
    Thus, the results presented for the category $\Cbar = \C / \I$ with $\I$ some $\mathfrak{P}$-complete ideal for a given valid partition $\mathfrak{P}$ also apply to the category $\C \cong \C / \langle 0 \rangle$.
        
\end{enumerate}
\end{remark}

\begin{lemma} \label{comparing I and I hat}
    Let $f: x \to y$ be a morphism in $\C$. Then $f$ belongs to $\Ihat$ if and only if there are $\sigma_1, \sigma_2 \in \Sigma_\mathfrak{P}$ such that $\sigma_2 f \sigma_1 \in \I$.  In particular, $f \in \Ihat$ if and only if $f$ is $\mathfrak{P}$-equivalent to a morphism in $\I$.
\end{lemma}

\begin{proof}
    The sufficiency is clear, because $f$ and $\sigma_2 f \sigma_1$ are $\mathfrak{P}$-equivalent.
    For the necessity, assume that $f \in \Ihat$. Then
    $$f = y_1 f_1 x_1 + \cdots + y_r f_r x_r$$
    where each $f_i$ is $\mathfrak{P}$-equivalent to a morphism $g_i \in \I$. Let $s, t$ be the respective domain and codomain of all $y_if_ix_i$. 
    For each $i$, there are $\sigma_{i1}, \sigma_{i2} \in \Sigma_\mathfrak{P}$ with $\sigma_{i2} f_i \sigma_{i1} \in \I$. We let $s' = {\rm min}(s, s(\sigma_{11}), \ldots, s(\sigma_{r1}))$ $t' = {\rm max}(t, t(\sigma_{12}), \ldots, s(\sigma_{r2}))$. We consider $\eta_{t't}f\eta_{ss'}$ to get
    $$\eta_{t't}f\eta_{ss'} = \eta_{t't}y_1 f_1 x_1\eta_{ss'} + \cdots + \eta_{t't}y_r f_r x_r\eta_{ss'}.$$
    Note that if $x_i$ is not in $\Sigma_\mathfrak{P}$, then we get the factorization $x_i = \sigma_{i_1}x_i'$. If $x_i$ is in $\Sigma_\mathfrak{P}$, then $x_i\eta_{ss'}$ also factors through $\sigma_{i_1}$. We get similar statements considering the $y_i$ and $\eta_{t't}y_i$. This shows that every term $\eta_{t't}y_1 f_i x_1\eta_{ss'}$ factors through $\sigma_{i2}f_i\sigma_{i1}$, showing that $\eta_{t't}f\eta_{ss'} \in \I$.
\end{proof}

We now assume that our ideal $\I$ is $\mathfrak{P}$-complete. Recall that we seek a functor $\mathbf{r}: \Cbar \to \Cbar(\M)$ that is left inverse to the inclusion functor $\iota$ and that inverts morphisms in $\Sigma_\mathfrak{P}$.  We define the functor $\mathbf{r}$ as follows. For an object $x$ in $\Cbar$, we let $s_x$ be the unique element of $\M$ such that $x, s_x$ belong to the same cell of $\mathfrak{P}$. If $p: x \to y$ is a morphism, then there is a morphism $p': s_x \to s_y$ with the property that $p, p'$ are $\mathfrak{P}$-equivalent. We set $\mathbf{r}(x)=s_x$ and $\mathbf{r}(p) = p'$.

\begin{lemma}\label{lem:r iota is the identity on Cbar of M}
    Assume $\I$ is $\mathfrak{P}$-complete for some valid partition $\mathfrak{P}$ and $\mathbf{r}$ is given as above.
    Then
    \begin{enumerate}
        \item $\mathbf{r}$ is well-defined,
        \item $\mathbf{r}$ inverts morphisms in $\Sigma_{\mathfrak{P}}$, and
        \item $\mathbf{r}\iota$ is the identity on $\Cbar(\M)$.
    \end{enumerate}
\end{lemma}
\begin{proof}
(1). Suppose that $p: x \to y$ is a morphism and that $p', p'': s_x \to s_y$ are two morphisms that are $\mathfrak{P}$-equivalent to $p$. Then $p'-p''$ is $\mathfrak{P}$-equivalent to the zero morphism, so that $p' = p''$ in $\Cbar$. Therefore, the functor $\mathbf{r}$ is well defined on morphisms.
A routine check shows that $\mathbf{r}$ is functorial.

(2). If $x, y$ are in the same cell of $\mathfrak{P}$, then it is clear that $\mathbf{r}(\lambda \eta_{yx}) = \lambda\eta_{s_x}$ so that $\mathbf{r}$ inverts morphisms in $\Sigma_{\mathfrak{P}}$.

(3). It is quick to check that $\mathbf{r}\iota$ is the identity on $\Cbar(\M)$ by definition.
\end{proof}

\section{Localization and functors}~\label{sec:functors}

\subsection{Localization}
We recall that given a valid partition $\mathfrak{P}$, we let $\Sigma_\mathfrak{P}$ be the set of non-zero scalar multiples of the morphisms $\eta_{x,y}$ in $\Cbar$ where $x,y$ are in the same cell from $\mathfrak{P}$. We want to localize $\Cbar$ at  $\Sigma_\mathfrak{P}$. As we will see, this will impose some conditions on $\I$. We first recall calculus of (left) fractions and localizations. The reader is referred to \cite{K10} for more details.

\begin{definition}\label{def:axioms for localisation}
Let $\Sigma$ be a set of morphisms in a $k$-category $\mathcal{D}$. We say that $\Sigma$ admits a \emph{calculus of left fractions in} $\mathcal{D}$ if the following three axioms are satisfied:
\begin{enumerate}
    \item[(LF1)]  The set $\Sigma$ is closed under composition and contains all identity morphisms.
    \item[(LF2)] Each pair of morphisms $f:x \to y$ and $\sigma: x \to x'$ with $\sigma \in \Sigma$ can be completed to a commutative square
    $$\xymatrix{x \ar[r]^f \ar[d]^\sigma & y \ar[d]^{\sigma'} \\
    x' \ar[r]^{f'} & z}$$
    where $\sigma' \in \Sigma$.
    \item[(LF3)] If $f, g: x \to y$ are morphisms with $f \sigma = g \sigma$ for some $\sigma \in \Sigma$, then there exists $\sigma' \in \Sigma$ with $\sigma' f = \sigma' g$.
\end{enumerate}

\end{definition}

\begin{definition}
    Let $\Sigma$ be a set of morphisms in $\mathcal{D}$ admitting a calculus of left fractions in $\mathcal{D}$. Then there exists a $k$-linear category $\mathcal{D}[\Sigma^{-1}]$ with a $k$-linear functor $\pi: \mathcal{D} \to \mathcal{D}[\Sigma^{-1}]$ that inverts morphisms in $\Sigma$ with the universal property that for any $k$-linear functor $F: \mathcal{D} \to \mathcal{E}$ that inverts morphisms in $\Sigma$, there exists a unique functor $G: \mathcal{D}[\Sigma^{-1}] \to \mathcal{E}$ with $F = G\pi$. 
\end{definition}

The category $\mathcal{D}[\Sigma^{-1}]$ above can be constructed as follows. It has as objects the same as the objects of $\mathcal{D}$. Its morphisms are given by left fractions up to an equivalence relation. A \emph{left fraction} in $\Hom_{\mathcal{D}[\Sigma^{-1}]}(x,y)$ is a pair $(f, \sigma)$ of morphisms such that $f: x \to y', \sigma: y \to y'$ and $\sigma \in \Sigma$. Two left fractions $(f_1: x \to y', \sigma_1: y \to y')$ and $(f_2: x \to y'', \sigma_2: y \to y'')$ are equivalent if there exists an object $z$ and a diagram 
$$\xymatrix{& y' \ar[d]^{u} & \\ x \ar[ur]^{f_1} \ar[dr]_{f_2} \ar[r]^{f_3} & z & y \ar[ul]_{\sigma_1} \ar[dl]^{\sigma_2} \ar[l]_{\sigma_3} \\
& y'' \ar[u]^{v} &}$$ where all inner triangles commute and where $\sigma_3 \in \Sigma$. The functor $\pi: \mathcal{D} \to \mathcal{D}[\Sigma^{-1}]$ can be taken to be the functor that is the identity on objects and which sends $f: x \to y$ to the left fraction $(f, \id_y)$.

\begin{proposition} \label{localization complete}
    If $\I$ is $\mathfrak{P}$-complete, then the set $\Sigma_\mathfrak{P}$ admits a calculus of left (and right) fractions. 
\end{proposition}

\begin{proof}
    Assume that $\I$ is $\mathfrak{P}$-complete. We analyze all three axioms (LF1), (LF2) and (LF3). For (LF1), this follows from the facts that for $x \in \Cbar$, the morphism $\eta_{xx}$ is the identity morphism and that for $x \le y \le z$ in a given cell of $\mathfrak{P}$, we have $\eta_{zy}\eta_{yx} = \eta_{zx}$. 

For (LF2), let  $x'\stackrel{{\sigma}}{\leftarrow} x \stackrel{{\alpha}}\rightarrow y$ be two morphisms with $\sigma\in \Sigma_{\mathfrak{P}}$. We need to show that we can complete these morphisms to a commutative diagram \[\xymatrix{x \ar^{\alpha}[r] \ar_{\sigma}[d] & y \ar^{\sigma'}[d] \\
           x' \ar^{\alpha'}[r] & y'  }\] with $\sigma'\in \Sigma_{\mathfrak{P}}.$
If $\alpha$ does not lies in $\Sigma_{\mathfrak{P}}$, then as a morphism of $\C$, $\alpha$ factors through $\sigma$ and we can find a morphism $\alpha': x' \to y$ with $\alpha = \alpha' \sigma$ and we are done by setting $\sigma' = \id$. If $\alpha$ lies in $\Sigma_{\mathfrak{P}}$ and $x \le x' \le y$, then the above argument applies. If $x' > y$, then we simply take $y' = x'$, $\alpha' = \id$ and $\sigma'$ is taken to be the required multiple of $\eta_{x'y}$ so that $\sigma = \sigma'\alpha$.

 For (LF3), let $\alpha,\beta:x\to y \in \C$ such that $\alpha\sigma = \beta\sigma$ for $\sigma \in \Sigma_\mathfrak{P}$. This implies that $(\alpha - \beta)\sigma \in \I$. Since $(\alpha - \beta)\sigma$ and $\alpha - \beta$ are $\mathfrak{P}$-equivalent, this means that $\alpha - \beta \in \I$, hence that $\alpha = \beta$ in $\Cbar$. Thus, we can simply take $\sigma'$ to be the identity morphism and we get $\sigma' \alpha = \sigma' \beta$.
 
The statement for right fractions follows from dual arguments.
\end{proof}

\begin{remark}
    In the above statement, we may have that $\Sigma_\mathfrak{P}$ admits a calculus of left fractions in $\Cbar$ even if $\I$ is not $\mathfrak{P}$-complete. Let $(Q,\mathcal{P})$ be the $\mathbb{A}_2$ quiver with arrow $\alpha$ and assume that $\mathcal{P}_\alpha$ is an open real interval. Consider the valid partition $\mathfrak{P}$ with cells given by the two vertices of $Q$ together with the entire $\mathcal{P}_\alpha$.
    
    Set $\I$ be the ideal generated by all $\eta_{y,x}$ where $x < y$ in $\mathcal{P}_\alpha$. Note that $\I$ is not $\mathfrak{P}$-complete because $\Ihat$ is the ideal generated by all $\eta_{y,x}$ where $x \leq y$ in $\mathcal{P}_\alpha$. However, we can still localize $\Sigma_\mathfrak{P}$ in $\Cbar$ as one can check that all three axioms in Definition~\ref{def:axioms for localisation} are satisfied.

    Choose a sample $\M$ of $\mathfrak{P}$.
    Note that $\Cbar(\M)$ has three nonzero objects but $\Chat(\M)$ has only two nonzero objects.
    Thus, the categories $\Cbar(\M)$ and $\Chat(\M)$ are not equivalent.
\end{remark}

The following proposition tells us that if $\I$ is not $\mathfrak{P}$-complete but $\Sigma_\mathfrak{P}$ admits a calculus of left fractions in $\Cbar$, then we can still complete $\I$ to get $\Ihat$ and the localized categories $\Cbar[\Sigma_\mathfrak{P}^{-1}]$ and $\Chat[\Sigma_\mathfrak{P}^{-1}]$ become equivalent.

\begin{proposition}
    Assume that $\Sigma_\mathfrak{P}$ admits a calculus of left fractions in $\Cbar$. Then $\Cbar[\Sigma_\mathfrak{P}^{-1}]$ is equivalent to $\Chat[\Sigma_\mathfrak{P}^{-1}]$.
\end{proposition}

\begin{proof}
   Consider the localization functors $\pi_1: \Cbar \to \Cbar[\Sigma_\mathfrak{P}^{-1}]$ and $\pi_2 : \Chat \to \Chat[\Sigma_\mathfrak{P}^{-1}]$ where the first exists by our assumption and the second exists by Proposition \ref{localization complete}. Finally, consider the projection $p: \Cbar \to \Chat$. Since $\pi_2 p$ sends all morphisms of $\Sigma_\mathfrak{P}$ to isomorphisms, there exists $F: \Cbar[\Sigma_\mathfrak{P}^{-1}] \to \Chat[\Sigma_\mathfrak{P}^{-1}]$ with $F \pi_1 = \pi_2 p$. This functor $F$ is the identity on objects and sends the left fraction $(f, \sigma)$ in $\Cbar[\Sigma_\mathfrak{P}^{-1}]$ to the left fraction $(f,\sigma)$ in $\Chat[\Sigma_\mathfrak{P}^{-1}]$. Hence, $F$ is clearly full and dense. For the faithfullness, let $(f_1,\sigma_1), (f_2, \sigma_2)$
with $f: x \to y'$ and $g: x \to y''$ morphisms in $\C$ and $\sigma: y \to y', \sigma_2: y \to y''$ morphisms in $\Sigma_\mathfrak{P}$. We assume that the left fractions $(f_1, \sigma_1), (f_2, \sigma_2)$ are equivalent in $\Chat$. Hence, we get an object $z$ with a diagram $$\xymatrix{& y' \ar[d]^{u} & \\ x \ar[ur]^{f_1} \ar[dr]_{f_2} \ar[r]^{f_3} & z & y \ar[ul]_{\sigma_1} \ar[dl]^{\sigma_2} \ar[l]_{\sigma_3} \\
& y'' \ar[u]^{v} &}$$ of morphisms in $\C$
where all four inner triangles commute in $\Chat$ and $\sigma_3 \in \Sigma_\mathfrak{P}$. In particular, $uf_1 - vf_2 \in \widehat{I}$. By Lemma \ref{comparing I and I hat}, there exist $\sigma_4, \sigma_5 \in \Sigma_\mathfrak{P}$ such that $\sigma_5(uf_1-vf_2)\sigma_4 \in \I$. In the category $\Cbar[\Sigma_\mathfrak{P}^{-1}]$, we get
$(\sigma_5,\id)(uf_1 - vf_2, \id)(\sigma_4, \id) = (0,\id)$. Since $(\sigma_4,\id)$ and $(\sigma_5,\id)$ are isomorphisms, this yields that $(uf_1-vf_2,\id) = (0,\id)$. This gives us the diagram
$$\xymatrix{& z \ar[d]^{\alpha} & \\ x \ar[ur]^{uf_1-vf_2} \ar[dr]_{0} \ar[r]^{0} & w & z \ar[ul]_{\id} \ar[dl]^{\id} \ar[l]_{\alpha} \\
& z \ar[u]^{\alpha} &}$$
with $\alpha \in \Sigma_\mathfrak{P}$ where all four inner triangles commute in $\Cbar$. Observe that $\alpha(u\sigma_1 - v\sigma_2) \in \Ihat$. Hence, we can find $\beta, \gamma \in \Sigma_\mathfrak{P}$ with $\gamma\alpha(u\sigma_1 - v\sigma_2)\beta \in \I$. The third axiom of localizations yields a $\beta' \in \Sigma_\mathfrak{P}$ with
$\beta'\gamma\alpha(u\sigma_1 - v\sigma_2) \in \I$. Considering $\sigma:= \beta'\gamma\alpha$ and
the diagram
$$\xymatrix{& y' \ar[d]^{\sigma u} & \\ x \ar[ur]^{f_1} \ar[dr]_{f_2} \ar[r]^{\sigma f_3} & w' & y \ar[ul]_{\sigma_1} \ar[dl]^{\sigma_2} \ar[l]_{\sigma\sigma_3} \\
& y'' \ar[u]^{\sigma v} &}$$
in $\Cbar$, this yields that $(f_1, \sigma_1)$ and $(f_2, \sigma_2)$ are equivalent left fractions in  $\Cbar[\Sigma_\mathfrak{P}^{-1}]$.
\end{proof}

Now, the functor $\mathbf{r}$ is such that all morphisms in $\Sigma_{\mathfrak{P}}$ are sent to isomorphisms. Therefore, by the universal property of the localization, there is a unique functor $\Phi: \Cbar[\Sigma_{\mathfrak{P}}^{-1}] \to \Cbar(\M)$ such that $\Phi\pi = \mathbf{r}$. The table in Figure~\ref{fig:table of functors} will help us to follow the proofs in the next couple of pages.  The left part displays the functors we have defined so far while the right part is the corresponding induced functors on the categories of pointwise finite-dimensional representations.

\begin{figure}[h]
    \centering
   \begin{tabular}{ |c| }
  \hline
  \xymatrix@C=13ex{
    & \Cbar \ar@{->>}[dr] \\
      \Cbar(\M) \ar[ur]^-{\bar{\iota}} \ar@{->>}[dr] & & \Chat \ar[dr]^{\pi} \ar@/_1pc/[dl]_{\mathbf{r}} & \\
      & \Chat(\mathcal{M}) \ar[rr]^{\pi\iota} \ar[ur]_(.55){\iota} && \Chat[\Sigma_\mathfrak{P}^{-1}] \ar@/^1pc/[ll]^{\Phi}
  } \\ \hline
  \xymatrix@C=2ex{
        & \rpwf \Cbar \ar[dl]_-{\bar{\iota}^*} \\
      \rpwf (\Cbar(\M)) & & \rpwf\left(\Chat\right)  \ar[dl]^(.3){\iota^*}  \ar@{_(->}[ul] & \\
      & \rpwf \left(\Chat(\mathcal{M})\right) \ar@/^1pc/[ur]^{\mathbf{r}^*} \ar@/_1pc/[rr]_{\Phi^*}  \ar@{_(->}[ul] && \rpwf \left(\Chat[\Sigma_\mathfrak{P}^{-1}]\right) \ar[ll]_{(\pi\iota)^*=\iota^*\pi^*} \ar[ul]_{\pi^*}
  } \\
  \hline
\end{tabular}
    \caption{Table of functors. When $\I$ is $\mathfrak{P}$-complete, the quotient arrows  $\Cbar\twoheadrightarrow\Chat$ and $\Cbar(\M)\twoheadrightarrow\Chat(\M)$ become identity arrows and the diagrams collapse to the lower triangles.}
    \label{fig:table of functors}
\end{figure}

\begin{proposition}\label{prop:equivalence from localized to sample}
The functor $\pi\iota: \Chat(\mathcal{M}) \to \Chat[\Sigma_\mathfrak{P}^{-1}]$ is an equivalence with quasi-inverse $\Phi$.
\end{proposition}

\begin{proof}

Let $x$ be an object of $\Chat[\Sigma_\mathfrak{P}^{-1}]$, which is an object of $\C$. Then there exists $y \in \mathcal{M}$ such that $x, y$ belong to the same cell. By definition of $\Sigma_\mathfrak{P}^{-1}$, $\pi(x) \cong \pi (y)$, hence $x \cong \pi \iota(y)$, which proves that $\pi \iota$ is dense.

Let $f: x \to y$ be a morphism in $\Chat$ where $x, y \in \mathcal{M}$. Assume that $\pi(f)=0$. That means that the left fraction $(f, \id_y)$ is equivalent to the left fraction $(0: x \to y, \id_y)$. But this implies that there exists a morphism $u: y \to y'$ in $\Sigma_\mathfrak{P}$ such that $uf = 0$. In this case, we get $\lambda \eta_{y' y}f \in \I$ for $\lambda$ a non-zero scalar. Since $\I$ is $\mathfrak{P}$-complete, that means that $f \in \I$ as well, so $f = 0$.

Now, let us prove that $\pi\iota$ is full. Let $x, y \in \mathcal{M}$ and let $(f,g)$ be a left fraction where $f: x \to y', g: y \to y'$ with $g \in \Sigma_\mathfrak{P}$. We want to prove that there is a morphism $h: x \to y$ in $\Chat(\mathcal{M})$ such that $(f,g) = \pi \iota (h)$. Equivalently, we need to find a morphism $h: x \to y$ in $\C$ such that
the left fractions $(f, g)$ and $(h, \id_y)$ are equivalent in $\Chat$. Since there is a morphism $y \to y'$, we have $y \leq y'$ in some cell in $\oP$. If $x = y$, then it is clear that $f$ factors through $y$. If not, the morphism $f: x \to y'$ factors through any element $z \leq y'$ in that cell. In particular, $f$ factors through $y$.  Hence, there exists a morphism $f': x \to y$ such that $f = g'f'$ where $g': y \to y'$ is in $\Sigma_\mathfrak{P}$. Let $\lambda$ be a non-zero scalar such that $g' = \lambda g$. The diagram
$$\xymatrix{& y' \ar[d]^{\id_{y'}} & \\ x \ar[ur]^f \ar[dr]_{\lambda f'} \ar[r]^{f} & y' & y \ar[ul]_g \ar[dl]^{\id_y} \ar[l]_{g'} \\
& y \ar[u]^{g'} &}$$
shows that the left fractions $(f, g)$ and $(\lambda f', \id_y) = \pi\iota(\lambda f')$ are equivalent, showing that the wanted morphism $h$ is just $\lambda f'$. Hence, $\pi \iota$ is full, which shows that $\pi \iota$ is an equivalence.

Finally, using that $\mathbf{r}\iota = \id$, we see that $\Phi(\pi \iota) = \id$. Since $\pi \iota$ is an equivalence, $\Phi$ is also an equivalence and has to be a quasi-inverse to $\pi \iota$.
\end{proof}

\begin{corollary}\label{cor:equivalence from localized to sample on reps}
    Consider the functors $\pi\iota$ and $\Phi$ in Proposition~\ref{prop:equivalence from localized to sample}.
    Then,
    \begin{enumerate}
        \item $(\pi\iota)^*\circ\Phi^*$ is the identity on $\rpwf\widehat{C}(\M)$ and
        \item $\Phi^*\circ(\pi\iota)^*$ is an auto-equivalence on $\rpwf\widehat{C}[\Sigma^{-1}_{\mathfrak{P}}]$.
    \end{enumerate}
\end{corollary}

Proposition~\ref{prop:equivalence from localized to sample} gives us two functors:

$$\ResM:=\iota^*: \rpwf\left(\Chat\right) \to \rpwf\left(\Chat(\M)\right)$$
and\label{res and ind}

$$\IndP:=\mathbf{r}^* = \pi^* \circ \Phi^*: \rpwf\left(\Chat(\M)\right) \to \rpwf\left(\Chat\right)$$

where the first is the \emph{restriction} functor, while the second will be referred to as the \emph{induction} functor.

It is not hard to see that for $M \in \displaystyle\rpwf\left(\Chat(\M)\right)$, we have that $\IndP(M)$ is the representation in $\displaystyle\rpwf\left(\Chat\right)$ with $\IndP(M)(x) = M(s_x)$ where $s_x$ is the unique point of $\M$ with $x, s_x$ in the same cell in $\mathfrak{P}$. If $f:x \to y$ is any morphism, then there is a unique morphism $f': s_x \to s_y$ that is $\mathfrak{P}$-equivalent to $f$ and we have $\IndP(M)(f) = M(f')$. In particular, $\IndP(M)$ is constant on each cell, in the sense that for $x \le y$ in the same cell of $\mathfrak{P}$, we have that $\IndP(M)(\eta_{yx}) = \id$.

\begin{remark}
    Although the functors $\ResM$ and $\IndP$ are just $\iota^*$ and $\mathbf{r}^*$ respectively, it is useful to have these new notation since they clearly indicate their dependence on the sample and partition, respectively.  Note also that $\ResM\circ\IndP$ is equivalent to the identity since $\mathbf{r}\iota = \boldsymbol{1}$. 
\end{remark}

\begin{example}
    Consider Example \ref{running-ex}. We consider the following representation $M$. Consider a two dimensional vector space $W$ with basis $\{w_1, w_2\}$, and let $V = {\rm span}(w_1)$. We define $M$ such that $M(x)$ is $V$ or $W$ for all $x$ in $\C$. Consider the interval $I = (2,4)$ of $\mathcal{P}_\alpha=(0,4)$ and the interval $J = (-\infty, 0]$ of $\mathcal{P}_\beta=\mathbb{R}$. Then $M(x)=W$ if and only if $x \in I \cup J \cup \{b\}$. Otherwise, $M(x) = V$. For $x \le y$ in $\oP$, $M(\eta_{yx})$ is the canonical inclusion.  For $x \le y$ in $\Pbar_{\beta}$, $M(\eta_{yx})$ is the canonical projection. Finally, $M(\gamma)$ is also the canonical projection.

    Note that the corresponding valid partition $\mathfrak{P} = \mathfrak{P}_M$ is
    $$\{a,b,c,d\} \sqcup (0,2] \sqcup (2,4) \sqcup (-\infty, 0] \sqcup (0, +\infty)$$
    with $8$ cells. Now, let us pick sample points as follows:$$\mathcal{M} = \{a,b,c,d\}\cup \{a_1\in (0, 2],\, a_2 \in (2,4),\, b_1 \in (-\infty, 0],\, b_2 \in (0, + \infty)\},$$
    where $a_1,a_2\in \mathcal{P}_\alpha$ and $b_1,b_2\in \mathcal{P}_\beta$.
    Then $\C(\mathcal{M})$ is identified with the path category of the quiver
    $$\xymatrix{a \ar[r] & a_1 \ar[r] & a_2 \ar[r] & b \ar[r] \ar[d] & b_1 \ar[r] & b_2 \ar[r] & c\\
    &&& d}$$
    which is of type $\widetilde{E}_7$. As shown above, $\mathcal{C}(\mathcal{M})$ is equivalent to the localization $\C[\Sigma^{-1}_{\mathfrak{P}}]$.
\end{example}

\section{Second decomposition theorem}\label{sec:2ndDecom}

Let $(Q,\mathcal{P})$ be a thread quiver and $\C$ be the corresponding path category.
Let $\Cbar$ be the quotient of $\C$ by a (weakly admissible) ideal $\I$.

 In this section, we fix a noise free representation $Z$ in $\rpwf \Cbar$, and consider the corresponding valid partition $\mathfrak{P}_Z$ with a corresponding set $\mathcal{M}$ of sample points.
 Notice that, for each cell $X$ of $\mathfrak{P}_Z$ and $x\leq y\in X$, we have $Z(\eta_{yx})$ is an isomorphism and this is also true for any direct summand of $Z$. 

\begin{lemma}\label{lem:refinement and representation of C hat}
    Let $Z \in \rpwf\Cbar$ be quasi noise free (so that $\mathfrak{P}_Z$ is valid). Let $\mathfrak{P}$ be another valid partition and $\Ihat$ be the corresponding $\mathfrak{P}$-completion of $\I$. If $\mathfrak{P}$ refines $\mathfrak{P}_Z$, then $Z$ is a representation in $\displaystyle\rpwf\left(\Chat\right)$.
\end{lemma}

\begin{proof}
    Let $f$ be an arbitrary morphism in $\Ihat$. From Lemma \ref{comparing I and I hat}, there exists a morphism $h$ in $\I$ that is $\mathfrak{P}$-equivalent to $f$. We need to show that $Z(f)=0$. Hence, we are given that $Z(h)=0$ and there exist $\sigma_1, \sigma_2, \sigma_3, \sigma_4 \in \Sigma_\mathfrak{P}$ such that $\sigma_2 f \sigma_1 = \sigma_4 h \sigma_3$. This yields $Z(\sigma_2 f \sigma_1)=0$. By the choice of $\mathfrak{P}$, we have that $Z(\sigma_2)$ and $Z(\sigma_1)$ are isomorphisms. Thus, $Z(f)=0$.
\end{proof}

There is no converse to the preceding lemma.
For example, in the hereditary case (Section~\ref{sec:hereditary}) $\I$ is the 0 ideal and so the ideal $\Ihat$ is still zero and the partitions $\mathfrak{P}$ and $\mathfrak{P}_Z$ are unrelated.

Now, we may consider $Z$ as a representation of $\Chat$ and similarly for any direct summand of $Z$.
Thus, to decompose $Z$ we will work over $\Chat$. Before presenting the main theorem of this section, we need the following.

\begin{proposition} \label{PropPiStar}
    The functor $\pi^*: \rpwf\left(\Chat[\Sigma_{\mathfrak{P}}^{-1}]\right) \to \displaystyle\rpwf\left(\Chat\right)$ is fully faithful.
\end{proposition}

\begin{proof}
    Since $\iota^* \pi^*$ is an equivalence, we see that $\pi^*$ is faithful. We claim that $\iota^*$ is faithful on the image of $\pi^*$. Let $M,N$ be representations in $\rpwf \C[\Sigma_{\mathfrak{P}}^{-1}]$ and let $f: \pi^*(M) \to \pi^*(N)$ be such that $\iota^*(f)=0$. For $m \in \mathcal{M}$, we are given that $f_m=0$. Let $x$ be in the same cell as $m$. With no loss of generality, assume that $m < x$. We have that $f_x\pi^*(M)(\eta_{xm}) = \pi^*(N)(\eta_{xm})f_m$. Since $\pi$ inverts morphisms in $\Chat$ and that $\pi^*(M) = M \circ \pi$, we see that $\pi^*(M)(\eta_{xm})$ is an isomorphism. Since $f_m = 0$, the above equation yields $f_x=0$.
    This implies that $f_x=0$ for all $x \in \C$, so that $f=0$. This proves the claim. For proving that $\pi^*$ is full, let $M,N \in \rpwf(\C[\Sigma_\mathfrak{P}^{-1}])$ with a morphism $g: \pi^*(M) \to \pi^*(N)$. Since $(\pi \iota)^*$ is an equivalence, there is a morphism $f: M \to N$ such that $(\pi \iota)^*(f) = \iota^*(g)$ or $\iota^*(\pi^*(f)) = \iota^*(g)$. By faithfullness of $\iota^*$, we get $g = \pi^*(f)$. This finishes the proof  of the proposition.
\end{proof}

Note that an additive functor that is full and faithful preserves indecomposability and decompositions into direct sums. In particular, the functor $\IndP$ has these properties.

\begin{remark} \label{IndRes equals identity}
    Notice that the representation $Z$ in $\rpwf\Cbar$ (as shown above, also in $\displaystyle\rpwf\left(\Chat\right)$) is isomorphic to a representation in the image of $\IndP$. In fact, by the choice of our partition ($\mathfrak{P} = \mathfrak{P}_Z$) we have
    $\IndP \circ \ResM (Z) \cong Z$.
\end{remark}


\begin{theorem}\label{thm:second decomposition theorem}
    Let $Z$ be a noise free representation in $\rpwf\Cbar$, $\mathfrak{P}_Z$ the induced valid partition, and $\mathcal{M}$ a sample of $\mathfrak{P}_Z$.
    Then $Z$ decomposes into a direct sum of indecomposable representations coming from $\displaystyle\rpwf\left(\Chat(\M)\right)$ (i.e.\ all in the image of $\IndP$), unique up to isomorphism.
    Furthermore, if $\mathcal{M}$ is finite then $Z$ is a finite direct sum of indecomposables.
\end{theorem}
\begin{proof}
    Since $\Chat(\mathcal{M})$ is small and that $\ResM(Z)$ is pointwise finite-dimensional, we get that $\ResM(Z)$ is a
    direct sum of indecomposables, each with local endomorphism ring; see \cite[Page 30]{GR97}, \cite[Theorem 1.1]{BCB20}. Moreover, notice that this direct sum decomposition is unique, up to isomorphism, by the Krull--Remak--Schmidt--Azumaya theorem.
    Furthermore, if $\mathcal{M}$ is finite then $\ResM(Z)$ is a finite direct sum since $Z$ is finite-dimensional at each object in $\C$.

    Let $\bigoplus_\lambda \widetilde{Z}_\lambda$ be a direct sum decomposition of $\ResM(Z)$ such that each $\widetilde{Z}_\lambda$ is indecomposable.
    Then we have
    \begin{displaymath}
        Z\cong \IndP \circ \ResM(Z)\cong \IndP\left( \bigoplus_\lambda \widetilde{Z}_\lambda \right) \cong \bigoplus_\lambda \IndP\left(\widetilde{Z}_\lambda\right).
    \end{displaymath}
    Now, note that it follows from Proposition \ref{PropPiStar} that $\IndP$ is fully faithful, hence, preserves (localness of) endomorphism algebras. Therefore, each summand $\IndP(\widetilde{Z}_\lambda)$ has local endomorphism algebras and hence is indecomposable.
\end{proof}

\begin{remark}
    Note that the proof of Theorem \ref{thm:second decomposition theorem} gives us a way to get an explicit decomposition of $Z$ into a direct sum of indecomposables. When $\displaystyle\rpwf\left(\Chat(\M)\right)$ is well-understood (such as when $\Chat(\M)$ is a finite-dimensional algebra whose representations are all classified), then this gives us a way to decompose $Z$. 
\end{remark}

Recall that given a representation $M$ in $\rpwf \C$, or $\rpwf\Cbar$, or $\displaystyle\rpwf\left(\Chat\right)$, which are full subcategories of $\rpwf\C$ that is closed under direct sums and direct summands (Remark~\ref{rem:Cbar}), Theorem \ref{thm:noise and noise free decomposition} gives us a decomposition $M = M_{NF} \oplus M_N$ where $M_{NF}$ is noise free and $M_N$ is noise. In the next corollary, we refine this decomposition using Theorem~\ref{thm:second decomposition theorem}.

\begin{corollary}[to Theorems~\ref{thm:c-b},~\ref{thm:noise and noise free decomposition},~and~\ref{thm:second decomposition theorem}]\label{cor:big decomposition}
    Let $\C$ be the path category of a thread quiver $(Q, \mathcal{P})$, $\I$ be an ideal, and $M$ be a pointwise finite-dimensional representation of $\Cbar$. Then $M$ decomposes uniquely up to isomorphism into a direct sum of indecomposable noise free representations $\{Z_\lambda\}_\lambda$ plus a direct sum of indecomposable noise representations $\{V_{\alpha,\eta}\}_\eta$ for each arrow $\alpha$:
    \begin{displaymath}
        M \cong \left(\bigoplus_\lambda Z_\lambda \right)\oplus\left(\bigoplus_{\alpha\in Q_1} \bigoplus_{\eta} V_{\alpha,\eta} \right),
    \end{displaymath}
    where each indecomposable has local endomorphism ring.
    Moreover, if $Q$ is finite then there are finitely-many $Z_\lambda$ summands.
\end{corollary}
\begin{proof}
    By Theorem~\ref{thm:noise and noise free decomposition} we see that $M\cong M_{NF}\oplus M_{N}$.
    The decomposition of $M_{NF}$ into $\bigoplus_\lambda Z_\lambda$ follows from Theorem~\ref{thm:second decomposition theorem}.
    The decomposition of $M_{N}$ into $\bigoplus_{\alpha\in Q_1} \bigoplus_{\eta} V_{\alpha,\eta}$ follows from Theorem~\ref{thm:c-b}.
    By Theorems~\ref{thm:c-b},~\ref{thm:noise and noise free decomposition},~and~\ref{thm:second decomposition theorem}, we have that this decomposition is unique up to isomorphism and each indecomposable has local endomorphism ring.
\end{proof}

Recall quasi noise free representations and the full subcategory $\rqnf \Cbar$ of $\rpwf \Cbar$ whose objects are quasi noise free representations (Definition~\ref{def:quasi noise free}).

\begin{corollary}[to Corollary~\ref{cor:big decomposition}]
    Let $\C$ be the path category of a thread quiver $(Q,\mathcal{P})$ and $\I$ an ideal.
    The class of indecomposable objects in $\rpwf\Cbar$ is the same as the class of indecomposable objects in $\rqnf\Cbar$.
\end{corollary}

\section{Projective and injective objects}\label{sec:proj-inj}

In this section, we consider a thread quiver $(Q, \mathcal{P})$ with the corresponding path category $\C$, together with a weakly admissible ideal $\I$ of $\C$ (Definition~\ref{weakly admissible}), and we consider the quotient $\Cbar:=\C/\I$.
We construct a class of indecomposable projective and injective representations in $\rpwf \Cbar$. More specifically, we construct a projective representation $P_J$ and an injective representation $I_J$ in $\rpwf \Cbar$ for each interval $J$ of a given $\oP$ for $\alpha \in Q_1$. These representations are indecomposable or zero. With an additional assumption on $\Cbar$, we prove in Proposition \ref{ThmAllProj} that these yield all indecomposable projective and injective representations.

\medskip

 It is clear that for any object $z \in \Cbar$, the representable functor $\Hom(z,-)$ is projective indecomposable in $\rpwf \Cbar$. Indeed, it is projective in $\Rep\ \Cbar$, and the definition of $\I$ being weakly admissible (Definition~\ref{weakly admissible}) guarantees that this representation belongs to $\rpwf \Cbar$ and is indecomposable. 
 
 \medskip

Recall $\mathcal{P}_\alpha$ is the poset assigned to the arrow $\alpha$ and that $\oP$ is $\mathcal{P}_\alpha\cup \{s(\alpha),t(\alpha)\}$ where $s(\alpha)=\min \oP$ and $t(\alpha) = \max \oP$ are considered distinct even if $\alpha$ is a loop.

Now, let us fix an interval $J$ in some poset $\oP$ for $\alpha \in Q_1$.
We define a representation $P_J$ as follows.
We consider $J^{\rm op}$ to be the linearly ordered set obtained by reversing the order on $J$.
Let $y$ be any object of $\Cbar$.
We define $P_J(y) = \varinjlim_{J^{\rm op}}\Hom(j,y)$ where the direct limit is indexed by the directed set $J^{\rm op}$ and the morphisms are $\Hom(\eta_{j_2, j_1}, y): \Hom(j_2,y) \to \Hom(j_1, y)$ for $j_1 \le j_2$ in $J$ (or $j_2 \le j_1$ in $J^{\rm op}$).
Observe that the limit always exists given that $\Cbar$ is Hom-finite, since the dimensions of the vector spaces in the limit are all bounded above by the (finite) dimension of any $\Hom(j,y)$ where $j \in J$ and $j \le y$ when $y \in J$. 
By the universal property of a direct limit, we can define $P_J$ on morphisms of $\Cbar$. We note that as a functor, $P_J$ coincides with the direct limit of the representable projectives $\Hom(j,-)$ for $j \in J$. Now, we get that for $M \in \rpwf \Cbar$,
\begin{eqnarray*}
\Hom(P_J, M) & \cong & \Hom(\varinjlim_{J^{\rm op}} P_j, M)\\
& \cong & \varprojlim_{J^{\rm op}}\Hom(P_j, M)\\
& \cong & \varprojlim_{J^{\rm op}}M(j)
\end{eqnarray*}
which are all natural in $M$.
Note also that if $J$ has a minimal element $x$, then $P_J = \Hom(x,-)$ is just the representable projective at $x$ and $\varprojlim_{J^{\rm op}}M(j) = M(x)$ as expected.

Observe that we have a pointwise dual functor  $D = \Hom_k(-, k): \rpwf\Cbar \to \rpwf (\Cbar^{\rm op})$. Therefore, we can use this to define the representations $I_J$ for $J$ some interval in a given $\oP$. In other words, we consider the representation $P_{J^{\rm op}}$ of the opposite category $\Cbar^{\rm op}$ and let $I_J:=D(P_{J^{\rm op}})$, which is a representation in $\rpwf\Cbar$.

\begin{definition}
    Let $J,J'$ be two intervals in some $\oP$.  We say that
    \begin{enumerate}
        \item The intervals $J$ and $J'$ \emph{have the same start} if there is $x \in J \cap J'$ such that $J \cap [s(\alpha), x] = J' \cap [s(\alpha), x]$.
        \item The intervals $J$ and $J'$ \emph{have the same end} if there is $x \in J \cap J'$ such that $J \cap [x, t(\alpha)] = J' \cap [x, t(\alpha)]$.
    \end{enumerate}
\end{definition}

We observe that the relations ``have the same start'' or ``have the same end'' are  equivalence relations for the intervals in a given $\oP$. 

\begin{remark}
We remark that $P_J$ could be the zero representation. This happens if and only if for every $y \in J$ there is some $x \le y$ in $J$ with $\eta_{yx} \in \I$. This means that when $P_J$ is non-zero, there is some $y \in J$ such that for all $x \le y$ in $J$, we have $\eta_{yx} \not\in \I$. Hence, in this case, there is an interval $J'$ having the same start as $J$ such that the interval representation $M_{J'}$ is a non-zero quotient of $P_J$ in $\rpwf \Cbar$.
\end{remark}

\begin{proposition} \label{IndecProj}
 Let $J$ be an interval in some $\oP$. The representations $P_J$ and $I_J$ are projective and injective representations in $\rpwf \Cbar$, respectively, and both are indecomposable with local endomorphism ring, whenever they are non-zero.
\end{proposition}

\begin{proof}
We only prove that $P_J$ is projective, as the proof that $I_J$ is injective follows from the duality. 

Let $f:M_1 \to M_2$ be an epimorphism of pointwise finite-dimensional representations. Recall that $J$ is some interval in $\oP$. Consider the subcategory $\C_\alpha$ having as objects the objects in $\oP$ and with morphisms given by the scalar multiples of the $\eta_{yx}$ for $x \le y$ in $\oP$. Let $(M_i)_\alpha$ be the restriction of $M_i$ to this subcategory. In other words, $(M_i)_\alpha$ is the confinement of $M$ to $\oP$. This gives rise to $(M_i)_\alpha = M_{i1} \oplus M_{i2}$ where $M_{i1}$ is the direct sum of all direct summands $Z$ of $(M_i)_\alpha$ where $Z$ is an interval representation $Z = M_{I}$ where $I$ and $J$ intersect and $I$ either contains a lower bound for $J$ or start in the same way. Since $M$ is pointwise finite, the representation $M_{i1}$ is a finite direct sum of interval representations. Hence, there exists $x' \le x \in J$ where $x$ is contained in the support of all of the direct summands of $M_{i1}$ for $i=1,2$ and $M_{i2}(\eta_{xx'})=0$ for $i=1,2$. Restricting $M_{i1}$ further to $J$ gives a representation where all the $\eta_{zy}$ for $y \le z \le x$ in $J$ act as isomorphisms. We can use this to directly check that $\Hom(P_J, M_i) \cong M_{i1}(x)$ for $i=1,2$. 
Hence, $\Hom(P_J, -)$ and $\Hom(\Hom(x,-), -)$ agree on the $M_i$ and morphisms between them. This yields that $\Hom(P_J, f)$ is onto.

For the indecomposability of $P_J$, assume that $P_J$ is non-zero. Recall that only the start of $J$ matters for the definition of $P_J$.  We may assume that $J$ is not a single point, since in that case, $P_J$ is representable projective. It follows from the previous remark that we may assume that $J$ is such that $M_{J}$ is a non-zero quotient of $P_J$. We now replace $J$ by an interval having the same start but which will have better properties. Let $J'$ be an interval in $\oP$ having at least two elements and the same start as $J$ and a maximum $x$. We consider the interval $J'' = J' \backslash \{x\}$, which also has the same start as $J$, and is non-empty. Since $J, J''$ have the same start, $P_J = P_{J''}$. We therefore replace $J$ by $J''$ but still denote it $J$. 

In this case, we have a morphism $f: P_x \to P_J$ whose cokernel is $M_J$.
Note that it is possible $P_x\to P_J$ is the zero map.
In this case, $M_J\cong P_J$ and so $P_J$ is indecomposable.
Thus, suppose $M_J\not\cong P_J$.
Hence, we have a short exact sequence 
$$0 \to Z \stackrel{u}{\to} P_J \stackrel{v}{\to} M_J \to 0$$
where there is an epimorphism $P_x \to Z$. Let $\varphi: Z \to Z$ be a morphism and $\psi: P_x \to P_x$ be a corresponding lift. We know that $P_x$ has finite-dimensional local endomorphism algebra. It is not hard to check that if $\psi$ is an isomorphism, then $\varphi$ is surjective and hence an isomorphism, and that if $\psi$ is nilpotent, then so is $\varphi$. The latter implies that $Z$ has local endomorphism ring.

Assume that there is a non-trivial decomposition $P_J = P_1 \oplus P_2$ so that with respect to this decomposition, we have $u = (u_1, u_2)^T$ and $v = (v_1, v_2)$.
Since $\Hom(P_x, M_J)= M_J(x)=0$, it follows that $\Hom(Z, M_J)=0$.
Hence $v_1u_1 = v_2u_2=0$.
Since the morphism $(u_1, 0)^T:Z \to P_1 \oplus P_2$ compose to zero with $v$, it follows that there exists an endomorphism $w: Z \to Z$ such that $(u_1, 0)^T = (u_1w, u_2w)^T$.
Since $Z$ has local endomorphism ring (which is finite-dimensional), either $w$ is an isomorphism or is nilpotent.
In the first case, we get $u_2=0$ and in the second case, we get $u_1=0$. With no loss of generality, assume that $u_2=0$.
Similarly, since $M_J$ also has local endomorphism ring, we also get that either $v_1=0$ or $v_2 = 0$.
Clearly, the only possibility is that $v_1=0$.
In this case, $Z \cong P_1 \cong P_x$ and $M_J \cong P_2$.
This implies that $P_x$ is a direct summand of $P_J$.
But we may repeat this argument by making the interval $J$ shorter and prove that for any $y$ in $J$, $P_y$ is a direct summand of $P_J$, which is impossible by the Krull--Remak--Schmidt--Azumaya theorem and since $P_J$ is pointwise finite-dimensional.
This proves that $P_J$ is indecomposable.
Recall that any pointwise finite-dimensional indecomposable representation has local endomorphism algebra; see for instance \cite[Page 32, Remark 5]{GR97}.
\end{proof}

\begin{remark}
    We note that if $J$ does not have a minimal element, although $P_J$ is projective in $\rpwf \Cbar$, it is not projective in $\Rep\ \Cbar$ when it is non-zero. Indeed, we have an epimorphism
    \[\bigoplus_{j \in J}P_j \to P_J.\]
    If $P_J$ were projective, then we would get that $P_J$ is a direct summand of $\oplus_{j \in J}P_j$, where the latter is a direct sum of representations with local endomorphism rings. Since $P_J$ is indecomposable, it follows from Theorem 6 in \cite{Warfield} that $P_J$ has to be isomorphic to one $P_j$, which is impossible.
    See also, \cite[Proposition 2.5.3]{IRT22}.
\end{remark}

\begin{proposition}
Consider two intervals $J, J'$ in some $\oP$ such that $P_J$ and $P_{J'}$ are non-zero. The projective representations $P_J, P_{J'}$ (resp. injective representations $I_J, I_{J'}$) are isomorphic if and only if $J,J'$ have the same start (resp. have the same end).
\end{proposition}

\begin{proof}
We only prove the projective case, the other being dual. The sufficiency is clear and follows from the properties of a direct limit. Now, suppose that $P_J$ is isomorphic to $P_{J'}$. Consider an interval $R$ that has the same start as $J$ so that we have a non-zero epimorphism $P_J \to M_R$. The isomorphism $P_{J'} \to P_J$ composed with this non-zero epimorphism yields that $J', R$ intersect. In fact, either they have the same start or $R$ contains a lower bound for $J'$. But since $R$ can be taken arbitrarily small having the same start as $J$, this means that $J', R$ have the same start.
\end{proof}

Next, we would like to describe all indecomposable projective representations and all indecomposable injective representations. To do that, we require the following.

\begin{definition}\label{def:Q bounded}
We call the category $\Cbar$ 
\begin{enumerate}
    \item  \emph{left $Q$-bounded} if, for each $x \in Q_0$, there are at most finitely many $y \in Q_0$ with $\Hom(y,x)$ non-zero. 
    \item \emph{right $Q$-bounded} if, for each $x \in Q_0$, there are at most finitely many $y \in Q_0$ with $\Hom(x,y)$ non-zero. 
    \item \emph{$Q$-bounded} if it is both left and right $Q$-bounded.
\end{enumerate}
    
\end{definition}

\begin{remark}
    When $Q$ is locally finite, the Gabriel--Roiter definition of admissibility for $kQ$ (see \cite[pg. 81]{GR97}) is equivalent to $kQ$ being a bounded $k$-category.

    When $\I=0$ and $\C=\Cbar$ is left (or right) $Q$-bounded, this is equivalent to the condition that $Q$ is left (or right) rooted with a finite ordinal (see, for example, \cite{CIT13}).
\end{remark}

\begin{theorem} \label{ThmAllProj} We have the following.
\begin{enumerate}
    \item Let $\Cbar$ be left $Q$-bounded. If $P \in \rpwf \Cbar$ is projective indecomposable, then there exists an arrow $\alpha$, and interval $J$ in $\oP$ such that $P$ is isomorphic to $P_J$.
    \item Let $\Cbar$ be right $Q$-bounded. If $I \in \rpwf \Cbar$ is injective indecomposable, then there exists an arrow $\alpha$, and interval $J$ in $\oP$ such that $I$ is isomorphic to $I_J$.
\end{enumerate}

\end{theorem}

\begin{proof}
We only prove $(1)$, since the proof of $(2)$ is dual. According to our first decomposition theorem, either $P = M_I$ for some interval $I$ contained in a given $\mathcal{P}_\alpha$, or else, $P$ is noise free. In the first case, we have an epimorphism $P_I \to P$, which has to split, showing that $P \cong P_I$. Let us assume the second case. Since $P$ is noise free, for each arrow $\alpha \in Q_1$, the set $\mathcal{P_\alpha}$ can be partitioned into finitely many intervals such that $P(\eta_{yx})$ is an isomorphism whenever $x,y$ belong to a given interval of that partition. Using this, it is not hard to see that for each arrow $\alpha$, there are finitely many intervals $I_{\alpha, 1}, \ldots, I_{\alpha, r_\alpha}$ of $\overline{\mathcal{P}_\alpha}$ such that we have an epimorphism
$$\bigoplus_{\alpha \in Q_1} \bigoplus_{i=1}^{r_\alpha}P_{I_{\alpha, r_\alpha}} \to P$$
where the summands in the direct sum are all indecomposable by Proposition \ref{IndecProj}.
The fact that $\Cbar$ is left $Q$-bounded guarantees that the left-hand side is pointwise finite, hence the above is an epimorphism in $\rpwf \Cbar$. This epimorphism splits and by Krull--Remak--Schmidt--Azumaya, it follows that $P$ has to be isomorphic to some $P_{I_{\alpha, r_\alpha}}$.
\end{proof}

\begin{example} Consider the path category $\Cbar$ from Example~\ref{ex:with ideal}(1). We point out that even though $\Cbar$ is not a hereditary category but it still has nice projective resolutions. 
Let us look at the projective resolution of the simple at $0$:
\begin{displaymath}
    \xymatrix{
        P_{[1,1]} \ar@{^(->}[r] & P_{(\frac{1}{2},1]} \ar[r] & P_{[\frac{1}{2},1)} \ar[r] & P_{(0,\frac{1}{2}]} \ar[r] & P_{[0,\frac{1}{2})} \ar@{->>}[r] & S_0.
    }
\end{displaymath}
    
\end{example}

\section{Examples and applications}\label{sec:exs}

The results we have obtained so far allow us to understand the indecomposable objects in $\rpwf \C$ or $\rpwf \Cbar$ provided we understand the indecomposable representations over subcategories of $\C$ or $\Cbar$ coming from valid partitions of the objects of $\C$, which we denote by $\C_0$. We note that categories of the form $\Cbar$ or $\C$ include finite-dimensional algebras given by quivers with an admissible ideal, locally discrete thread quivers as studied by Berg and Van Roosmalen in \cite{BvR14}, arbitrary linearly ordered sets with a minimum and maximum, and the real line (by taking a linearly oriented quiver of type $\mathbb{A}_\infty^\infty$ and threading each arrow with an open real interval).

Recall that $k$ is an algebraically closed field. In the following subsections, we will see examples of our theory we introduce in the present paper.

\subsection{New representation types for thread quivers.}

The following definitions come naturally as a consequence of our decomposition theorems. We assume here that $Q$ is a finite quiver and we consider the corresponding category $\Cbar:=\C / \I$ where $\I$ is weakly admissible. We note that if $\mathcal{M}$ is any set of sample points coming from a valid partition, then $\Cbar(\mathcal{M})$ is equivalent to a finite-dimensional algebra.





\medskip

Note that if $M \in \rpwf\Cbar$ is indecomposable, then we can consider the valid partition $\mathfrak{P}=\mathfrak{P}_M$ and the corresponding $\mathfrak{P}$-completion $\Ihat$ of $\I$ and note that $M$ is also annihilated by $\Ihat$, equivalently, $M$ is a representation over $\Chat$. In particular, if $\mathcal{M}$ is a sample of $\mathfrak{P}_M$, then $M$ is isomorphic to $\IndPM(\ResM M)$ (see Remark \ref{IndRes equals identity}) so is in the image of some induction functor. Therefore, every indecomposable representation in $\rpwf\Cbar$ is induced from an indecomposable representation of a category of the form  $\Chat(\mathcal{M})$. This leads to the following definition of representation types.

\begin{definition}\label{def:virtually finite and tame}
    We say that the category $\Cbar$ is of \emph{virtually tame representation type} (respectively, of \emph{virtually finite representation type}) if, for any valid partition $\mathfrak{P}$ of $\C_0$ with sample $\mathcal{M}$, then $\Chat(\mathcal{M})$ is of tame representation type (respectively, of finite representation type).
\end{definition}

\begin{example}
    \begin{enumerate}
        \item When $Q$ is of type $\mathbb{A}_n$ for some positive integer $n$, then $\C$ is of virtually finite representation type.
        \item When $Q$ is of acyclic type $\widetilde{\mathbb{A}}_n$ for some positive integer $n$, then $\C$ is of virtually tame representation type.
    \end{enumerate}
\end{example}

For other (extended) Dynkin types, some further restrictions on $\mathcal{P}$ are needed. We will see more examples later.
See Section~\ref{sec:hereditary} for more discussion.

There is another natural way to define the representation type of a category of the form $\Cbar$. An element $d \in \Hom(\Cbar_0, \mathbb{Z})$ is a \emph{dimension vector} provided $d(x) \ge 0$ for all $x \in \Cbar_0$. 

Next, to each dimension vector $d$, we attach a partition $\mathfrak{P}_d$ of $\C_0$ as follows. Let $\alpha$ be an arrow of $Q$.
Using the definition of a dimension vector, we have a map $d_{\alpha}:\mathcal{P}_\alpha\to \mathbb{N}$.
For $n\in \mathbb{N}$, consider $( d_{\alpha} )^{-1} (n)\subseteq \mathcal{P}_\alpha$.
Take $\{X_{\lambda,n,\alpha}\}_{\Lambda_n}$ to be the partition of $( d_{\alpha} )^{-1} (n)$ into maximal intervals of $\mathcal{P}_\alpha$ so that $\coprod_{\lambda\in\Lambda_n} X_{\lambda,n,\alpha}=(d_\alpha)^{-1}(n)$.
That is, any partition of $( d_{\alpha} )^{-1} (n)$ into intervals of $\mathcal{P}_\alpha$ is a refinement of $\{X_{\lambda,n,\alpha}\}_{\Lambda_n}$.
Then $\bigcup_{n\in\mathbb{N}} \{X_{\lambda,n,\alpha}\}_{\Lambda_n}$ is a partition of $\mathcal{P}_\alpha$ into intervals.
Doing this for each arrow $\alpha$, and putting each vertex $i$ from $Q_0$ into its own cell, yields a partition $\mathfrak{P}_d$ of $\C_0$:
\begin{displaymath}
    \mathfrak{P}_d := \{\{x\} | x\in Q_0\} \cup \left( \bigcup_{\alpha\in Q_1} \bigcup_{n\in\mathbb{N}} \{X_{\lambda,n,\alpha}\}_{\Lambda_n} \right).
\end{displaymath}

\begin{definition}\label{def:valid dimension vector}
    Let $d$ be a dimension vector. We call $d$ \emph{valid} if $\mathfrak{P}_d$ is valid.
\end{definition}

Let $M \in \rpwf\Cbar$ and let $d_M$ be its dimension vector.
Then $\mathfrak{P}_M$ is a refinement of $\mathfrak{P}_{d_M}$.
Moreover, if $\mathfrak{P}_M$ is valid then $d_M$ is valid (and so is $\mathfrak{P}_{d_M}$). 
In particular, if $M$ is indecomposable then it follows from our decomposition theorem that $\sum_{n\in\mathbb{N}} |\Lambda_n|~<~\infty$.
Thus all of $\mathfrak{P}_M$,\ $d_M$, and $\mathfrak{P}_{d_M}$ are valid.

\medskip

Having defined (valid) dimension vectors, we can now consider other variations of representation type as follows.
In the definition below, there is no loss of generality in working with valid dimension vectors since each pointwise finite-dimensional indecomposable representation has a valid dimension vector.

\begin{definition}\label{def:essentially finite and tame}
\begin{enumerate}
    \item We say that $\Cbar$ is of \emph{essentially finite representation type} if, for any valid dimension vector $d$, there are at most finitely many non-isomorphic indecomposable representations of dimension vector $d$.
    \item We say that $\Cbar$ is of \emph{essentially tame representation type} if, for any valid dimension vector $d$, all but finitely many indecomposable representations, up to isomorphism, fall into finitely many one-parameter families of $d$-dimensional indecomposable representations. 
\end{enumerate}
\end{definition}

In the above definition, by a one-parameter family $\mathcal{F}$
with valid dimension vector $d$, we mean that there is a sample $\mathcal{M}$ of $\mathfrak{P}_d$ and a one-parameter family $\mathcal{F}'$ of indecomposable representations of $\Chat(\mathcal{M})$ such that every representation in $\mathcal{F}$ is isomorphic to a representation in $\IndP(\mathcal{F}')$.

\begin{proposition}\label{prop:essentially is virtually}
    If the category $\Cbar$ is of essentially finite (respectively essentially tame) representation type, then it is of virtually finite (respectively virtually tame) representation type.
\end{proposition}

\begin{proof} Assume that $\Cbar$ is not virtually finite. That means that there is a valid partition $\mathcal{M}$ of $\C_0$ such that $\Chat(\mathcal{M})$ is not of finite representation type. By the second Brauer-Thrall conjecture (now Theorem in~\cite{B85}), there is a dimension vector $d$ for $\Chat(\mathcal{M})$ for which $\Chat(\mathcal{M})$ admits infinitely many non-isomorphic indecomposable representations of dimension vector $d$.

Now, the functor $\IndP:~\rpwf(\Chat(\mathcal{M}))~\to~\displaystyle\rpwf\left(\Chat\right)$ applied to this infinite family yields an infinite family of non-isomorphic indecomposable representations in $\rpwf (\Chat)$ all having the same dimension vector, and thus also in $\rpwf \Cbar$. Hence $\Cbar$ is not essentially finite. The proof for the tame case is very similar and we omit it.
\end{proof}

We note that the converse of the above proposition is not true, both for the finite type and tame type. The following example illustrates this.

\begin{example}~\label{essentially/virtually examples}

\begin{enumerate}
    \item In the following table, by $\mathbb{A}_n$ or $\widetilde{\mathbb{A}}_n$ we mean any path category of an acyclic quiver of that type. By $\Lambda_1$, we mean that $\C$ is the path category of a thread quiver of type $\mathbb{A}_n$ with continuous intervals as the threadings. The category $\Lambda_2$ is the result of threading both arrows of the Kronecker quiver with the real interval $(0,1)$, picking two points $x<y$ in the top arrow, setting $\I=\langle\eta_{yx}\rangle$, and taking $\Lambda_2=\Cbar$. Now, $\Lambda_3$ is the result of threading the middle arrow of a quiver of type $\widetilde{\mathbb{D}}_5$ (which has 6 vertices) with the real interval $(0,1)$. Finally, $\Lambda_4$ is the result of threading the two arrows of the Kronecker quiver with the real interval $(0,1)$. 

\begin{center}
{\tabulinesep=1.2mm
\begin{tabu}{|r|c|c|c|}
    \hline
    & classically & essentially & virtually \\ \hline
    finite & $\mathbb{A}_n$ & $\Lambda_1$ & $\Lambda_2$ \\ \hline
    tame & $\widetilde{\mathbb{A}}_n$ & $\Lambda_3$ & $\Lambda_4$ \\ \hline
\end{tabu}}
\end{center}
Below are the thread quivers of these $4$ examples, where a double-arrow means that the linearly ordered set on top is the open interval $(0,1)$.
\begin{align*}
   \Lambda_1 :& \ \xymatrix{\bullet_1 \ar@{=>}[r] & \bullet_2 \ar@{=>}[r] & \cdots \ar@{=>}[r] & \bullet_{n-1} \ar@{=>}[r] & \bullet_n } \\
   \Lambda_2, \Lambda_4 :&  \xymatrix{\bullet_a\ar@/^2ex/@{=>}[r]^{~} \ar@/_2ex/@{=>}[r]_{~} & \bullet_b }\\
   \Lambda_3 :&  \xymatrix@R=1ex{ \bullet \ar[dr] & & & \bullet \\ 
     & \bullet \ar@{=>}[r]  & \bullet \ar[ur] \ar[dr] & \\
               \bullet \ar[ur]  & & &  \bullet} 
\end{align*}

For $\Lambda$ some entry in the table above, it also has the properties of being in the same row in all the columns to its right.
For example, $\Lambda_3$ is also virtually tame.
However, $\Lambda_2$ is not of essentially finite type while $\Lambda_4$ is not of essentially tame type. For $\Lambda_2$ and $\Lambda_4$, the distinction between the essential and virtual types can be seen in the analysis of the (valid) dimension vector consisting of all ones.
Each entry does \emph{not} have the properties of the entries to its left.

\item In Figure \ref{reptype}, we see that threading can change the representation type. In this figure, the thicker arrows represent threading with an open real interval and thin arrows represent no (or empty) threading. Note that (I) is representation finite, that (II) is essentially finite (therefore virtually finite) while (III) and (IV) are wild (not essentially tame). 

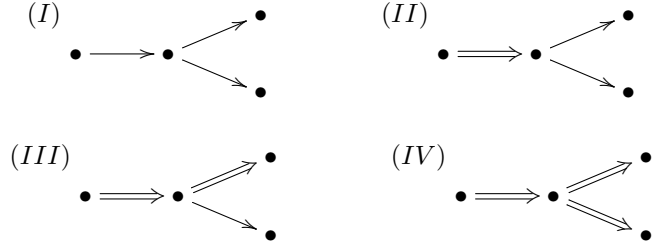
\begin{figure}[H]
    \label{reptype}
\[(I) \xymatrix@R=1ex{   & & \bullet \\ 
    \bullet \ar[r] & \bullet \ar[ur] \ar[dr] & \\
               & & \bullet} \qquad\qquad (II) \xymatrix@R=1ex{   & & \bullet \\ 
    \bullet \ar@{=>}[r] & \bullet \ar[ur] \ar[dr] & \\
               & & \bullet}\] 

\[(III) \xymatrix@R=1ex{   & & \bullet \\ 
    \bullet\ar@{=>}[r] & \bullet\ar@{=>}[ur] \ar[dr] & \\
               & & \bullet} \qquad\qquad (IV) \xymatrix@R=1ex{               & & \bullet \\ 
    \bullet\ar@{=>}[r] & \bullet\ar@{=>}[ur] \ar@{=>}[dr] & \\
               & & \bullet}\] 
               \caption{Representation type varies depending on how we thread.}
\end{figure}

\end{enumerate}
\end{example}

\begin{remark}
    We note that in the virtually finite representation type case, each indecomposable pointwise finite-dimensional representation is induced from an indecomposable representation of a subcategory which is of finite representation type. If one considers posets, then this idea was also exploited in \cite{ABH24}. 
\end{remark}

\subsection{Special biserial categories from thread quivers} 
\label{SubSecSB}

Before we define special biserial categories from thread quivers, we first define two types of families of relations that we call quadratic relations.

\begin{definition}\label{def:quadratic}
    Given an object $x$ in $\C$, either $x$ lies in some $\mathcal{P}_\alpha$ for some arrow $\alpha$ or else $x\in Q_0$. In the latter case, we assume that we have an incoming arrow $\alpha$ to $x$ and an outgoing arrow $\beta$ from $x$. In each case, we define two types of families of relations that we call \emph{quadratic}.
    \begin{enumerate}
        \item[(I)]  Let $x\in Q_0$.
        For $y \in \mathcal{P}_\alpha\cup\{s(\alpha)\}$ and any $z \in \mathcal{P}_\beta\cup\{t(\beta)\}$, we consider the family $\{\eta_{zx}\eta_{xy}\}_{y,z}$ and we call it the \emph{quadratic family associated to $\beta\alpha$}. Note that if $\alpha$ is a loop, then by $\eta_{x,s(\alpha)}$, we mean the arrow $\alpha$. Similarly, if $\beta$ is a loop, then by $\eta_{t(\beta),x}$, we mean the arrow $\beta$. 
        \item[(II)] Let $x\in \mathcal{P}_\alpha$. For any $y<x<z$ with $y,z\in\overline{\mathcal{P}}_\alpha$, we consider the family $\{\eta_{zx}\eta_{xy}\}_{y,z}$ that we call the \emph{quadratic family associated to $x$}.
    \end{enumerate}
\end{definition}

We let $Q$ be a \emph{biserial quiver}, that is, a finite quiver such that for each vertex $v$, there are at most two arrows leaving $v$ and at most two arrows entering $v$. We consider a thread quiver $(Q, \mathcal{P})$ and let $\C$ denote the corresponding path category. Finally, we consider a weakly admissible ideal $\I$ with the properties that
\begin{itemize}
    \item[(SB1)] Let $\alpha, \beta, \gamma \in Q_1$ such that $v:=t(\alpha) = s(\beta) = s(\gamma)$ and $\beta\neq\gamma$. Then $\I$ contains at least one of the quadratic family associated to $\beta\alpha$ or the quadratic family associated to $\gamma\alpha$.

    \item[(SB2)] Let $\alpha, \beta, \gamma \in Q_1$ be distinct such that $v:=t(\alpha) = t(\beta) = s(\gamma)$ and $\alpha\neq \beta$.  Then $\I$ contains at least one of the quadratic family associated to $\gamma\alpha$ or the quadratic family associated to $\gamma\beta$.

\end{itemize}

Note that one may have infinitely many relations in $\I$. Now, we consider the quotient $\Cbar = \C /\I$ and call this a \emph{special biserial} category of $(Q, \mathcal{P})$. 

Notice that Definition~\ref{def:quadratic} allows the possibility that we have not threaded any arrows. In this case, the special biserial categories are special biserial algebras in the classical sense. We can similarly define gentle and string categories. We make the following comments.

\begin{remark}
   We define gentle categories and string categories, which are analogous to gentle algebras and string algebras, respectively.
   \begin{enumerate}
       \item For gentle categories, we replace ``at least one'' in each of (SB1), (SB2) by ``exactly one''. Moreover, any other (family of) relations in $\I$ must be quadratic (Definition~\ref{def:quadratic}).
       \item For string categories, we require, in addition to (SB1) and (SB2), that the generators of $\I$ are all nontrivial path-like elements.
   \end{enumerate}
\end{remark}

We note that if $\M$ is a set of sample points from a valid partition, then $\Cbar(\M)$ is a finite-dimensional special biserial (respectively gentle, string) algebra, provided $\Cbar$ is a special biserial (respectively gentle, string) category. Since special biserial, gentle and string algebras are all of tame representation type, this leads to the following.

\begin{proposition}\label{special biserial, gentle, string}
    Let $\Cbar$ be a special biserial (or gentle, or string) category from a thread quiver $(Q, \mathcal{P})$ where $Q$ is a biserial quiver. Then $\Cbar$ is virtually tame.
\end{proposition}

\begin{proof}
    Let $\mathfrak{P}$ be a valid partition with corresponding sample $\M$, and consider $\Chat(\M)$. We can present this algebra using quiver $\Gamma$ where $\Gamma$ is obtained from $Q$ by replacing each arrow by a finite linear quiver of type $\mathbb{A}$. Clearly, $\Gamma$ is biserial. We get $\Chat(\M) \cong k\Gamma / I$. Let us describe $I$, which will be induced from $\Ihat$. We first consider a quadratic family of Type II associated to $x$.
    \begin{enumerate}
        \item Suppose $\{x\}$ is a cell of $\mathfrak{P}$ ($x$ is both the maximum and the minimum of its cell). In this case, the quadratic family induces one quadratic relation going through $x$.
        \item Suppose either (i) $x$ is the maximum of a cell $X$ but not the minimum or (ii) $x$ is the minimum but not the maximum of a cell $X$. Let $u$ be in $\Gamma$ corresponding to cell $X$. Then $u$ has one incoming arrow $\alpha$ and one outgoing arrow $\beta$. If $x$ is the maximum of $X$, then $\beta \in I$, while if $x$ is the minimum of $X$, then $\alpha \in I$.
        \item Suppose $x$ is neither the maximum nor the minimum of the cell $X$ that contains it. If $u$ is the vertex of $\Gamma$ associated to $X$, then the idempotent $e_u$ belongs to $I$.
    \end{enumerate}
    For a quadratic family of Type I, we recover a quadratic relation as in case (1) above. Hence, the algebra $\Chat(\M)$ is isomorphic to $k\Gamma' / I'$ where $\Gamma'$ is obtained from $\Gamma$ by possibly removing vertices (case (3) above) and arrows (case (2) above), and $I'$ is the induced admissible ideal from $I$. It is clear that $(\Gamma', I')$ is a special biserial bound quiver with $\Chat(\M) \cong k\Gamma' / I'$. Therefore, since all biserial algebras are tame \cite{WW85}, we conclude that $\Cbar$ is virtually tame.
\end{proof}

\begin{remark}\label{rmk:special biserial representations}
    It follows from Theorem~\ref{thm:second decomposition theorem} and Proposition~\ref{special biserial, gentle, string} that we have a classification of the representations of our given special biserial (or gentle, or string) category $\Cbar$ as follows. For every valid partition $\mathfrak{P}$ and each sample $\M$ coming from $\mathfrak{P}$, we consider all $\IndP(Z)$ for each string or band representation $Z$ of $\rpwf \Chat(\M)$. 
\end{remark}

\begin{example}
    Notice that Example~\ref{ex:with ideal}(2) is a gentle category. For any sample $\mathcal{M}$, we get a representation-finite gentle algebra equivalent to $\Cbar(\mathcal{M})$.
\end{example}

\section{Some subcategories and new hereditary categories}\label{sec:hereditary}

Notice that $\Rep\ \C$ is not hereditary, in general. For instance, consider $Q=\mathbb{A}_2$ with a copy of $\mathbb{R}$ for $\mathcal{P}$. Take any real number $x$ and let $S_x$ be the corresponding simple representation. Then the kernel of the epimorphism $P_x \to S_x$ is not projective in $\Rep\ \C$. More generally, the same phenomenon happens if $x$ is any object of some $\mathcal{P}_\alpha$ having no immediate successor. This shows that path categories of thread quivers are different from path categories of quivers. Indeed, the category of representations of any quiver is always hereditary; see \cite[Page 79]{GR97}. This also yields that the category of pointwise finite-dimensional representations of a quiver is hereditary. We do not know if the latter holds for the path category of an arbitrary thread quiver. 

In this section, we prove that under some assumptions on $Q$, the category $\rpwf\C$ is hereditary. We also prove that the category of \emph{quasi noise free} representations is always hereditary and abelian.

\subsection{The general case.}

Recall the definition of a quasi noise free representation (Definition~\ref{def:quasi noise free}) and that $\rqnf\C$ denotes the full subcategory of $\rpwf\C$ consisting of the quasi noise free representations. We start with the following.

\begin{proposition}\label{thm:quasi noise free is serre}
    The category $\rqnf\C$ a is Serre subcategory of $\rpwf \C$. In particular, it is abelian.
\end{proposition}

\begin{proof}
     Let $M$ be a quasi noise free representation and $L$ be a subrepresentation of $M$. We first show that $L$ is quasi noise free. We let $\alpha \in Q_1$. Since $M$ is quasi noise free, the partition $\mathfrak{P}_M$ is valid (Proposition~\ref{prop:noise free inducess valid partition}). So, $\mathcal{P}_\alpha$ is partitioned into finitely many cells $X_1,\ldots,X_r$ such that for any $x\leq y\in X_i$ the map $M(\eta_{yx})$ is an isomorphism.
     For such $x,y$, the map $L(\eta_{yx})$ is the restriction of $M(\eta_{yx})$, hence is injective. But the dimension of $L$ on $X_i$ can take only finitely many values, in increasing order when we follow the order in $X_i$. This means that $X_i$ can be partitioned into finitely many subintervals such that for each such subinterval and $x < y$ in it, $L(\eta_{yx})$ is an isomorphism.  Hence, it follows that for such $x,y$, the map $(M/L)(\eta_{yx})$ is also an isomorphism. This proves that both $L, M/L$ are quasi noise free. Now, let
     \[0 \to M \to E \to N \to 0\]
     be a short exact sequence of pointwise finite-dimensional representations with $M,N$ quasi noise free. We are given two valid partitions $\mathfrak{P}_{M}$ and $\mathfrak{P}_{N}$.
     For each arrow $\alpha \in Q_1$, we can define a new partition of $\mathcal{P}_\alpha$ by intersecting the corresponding cells from $\mathfrak{P}_{M}$ and $\mathfrak{P}_{N}$.
     This yields the valid partition $\mathfrak{P}_{M\oplus N}$ originally defined on page~\pageref{partition of the direct sum}, which is a refinement of $\mathfrak{P}_E$.
     This proves that $E$ is quasi noise free.
\end{proof}

The next theorem yields a hereditary abelian category that contains all indecomposable pointwise finite-dimensional representations.

\begin{theorem}\label{thm:quasi noise free is hereditary}
    The category $\rqnf \C$ is hereditary.
\end{theorem}

\begin{proof}
    It is sufficient to prove that for a quasi noise free representation $M$, the functors $\Ext^1(M,-)$ and $\Ext^1(-,M)$ are right exact. We only consider the first case, as the other case is similar. Thus, we start with an epimorphism $g: Y \to Z$
    in $\rqnf \C$ and we need to prove that the map $\Ext^1(M,g): \Ext^1(M,Y) \to \Ext^1(M,Z)$ is surjective. Equivalently, we start with a non-split extension
    \[0 \to Z \stackrel{u}{\to} E \stackrel{v}{\to} M \to 0\]
    in $\rqnf \C$ and we need to find a quasi noise free representation $E'$ leading to a commutative diagram
    \[\xymatrix{0 \ar[r] & Y \ar[r] \ar[d]^g & E' \ar[r] \ar[d] & M \ar[r] \ar@{=}[d] & 0 \\
    0 \ar[r] & Z \ar[r]^u  & E \ar[r]^v  & M \ar[r] & 0}\]
Now, we consider the representation $T:=M \oplus E \oplus Z \oplus Y$ and take the corresponding valid partition $\mathfrak{P} = \mathfrak{P}_T$. We choose a set $\mathcal{M}$ of sample points and consider the functor $\ResM = \iota^*: \rpwf \C \to \rpwf \C(\mathcal{M})$. Since this functor is fully faithful exact, we get a short exact sequence
\[\xymatrix@C=8ex{0 \ar[r] & \ResM(Z) \ar[r]^-{\ResM(u)} & \ResM(E) \ar[r]^-{\ResM(v)} & \ResM(M) \ar[r] & 0}\]
with an epimorphism $\ResM(g): \ResM(Y) \to \ResM(Z)$. Since the category $\C(\mathcal{M})$ is equivalent to the category of pointwise finite-dimensional representations of a quiver, it follows from \cite[page 79]{GR97} that $\Ext^1(\ResM(M), \ResM(g))$ is surjective. Therefore, we have a commutative diagram
\[\xymatrix{0 \ar[r] & \ResM(Y) \ar[r] \ar[d]^{\ResM(g)} & U \ar[r] \ar[d] & \ResM(M) \ar[r] \ar@{=}[d] & 0 \\
    0 \ar[r] & \ResM(Z) \ar[r]^{\ResM(u)}  & \ResM(E) \ar[r]^{\ResM(v)}  & \ResM(M) \ar[r] & 0}\]
for some pointwise finite-dimensional representation $U$ in $\C(\mathcal{M})$. Applying the exact functor $\IndP$ to this diagram, and observing that by the choice of $\mathfrak{P}_T$, $\IndP \ResM(M) \cong M, \IndP \ResM(Z) \cong Z, \IndP \ResM(E) \cong E, \IndP \ResM(Y) \cong Y$, we get a commutative diagram
 \[\xymatrix{0 \ar[r] & Y \ar[r] \ar[d]^g & \IndP \ResM(U) \ar[r] \ar[d] & M \ar[r] \ar@{=}[d] & 0 \\
    0 \ar[r] & Z \ar[r]^u  & E \ar[r]^v  & M \ar[r] & 0}\]
    which shows that $\Ext^1(M,g)$ is surjective.
\end{proof}

\begin{example}
    Let us continue Example~\ref{ex:A2}(1).
    Let $M= \bigoplus_{n=1}^\infty M_{(1-\frac{1}{n},1)}$, which is pointwise-finite-dimensional.
    Notice that for each $M_{(1-\frac{1}{n},1)}$ and $M_{(1-\frac{1}{n'},1)}$, we have \[\Ext^1(M_{(1-\frac{1}{n},1)},M_{(1-\frac{1}{n'},1)})=0=\Ext^1(M_{(1-\frac{1}{n'},1)},M_{(1-\frac{1}{n},1)}).\]
    Also note that $\Ext^1(S_1,S_1)=0$ where $S_1$ is the simple representation at vertex $1$.
    Thus, both $M$ and $S_1$ are pointwise finite-dimensional and rigid.

    However, $\Ext^1(M_{(1-\frac{1}{n},1)},S_1)\cong \Bbbk$ for each $n\geq 1$ and so
    \begin{align*}
        \Ext^1(M, S_1) &= \Ext^1\left(\bigoplus_{n=1}^\infty M_{(1-\frac{1}{n},1)}\, ,\, S_1 \right) \\
        &= \prod_{n=1}^\infty \Ext^1\left(M_{(1-\frac{1}{n},1)}\, ,\, S_1 \right) \\
        &= \prod_{n=1}^\infty \Bbbk.
    \end{align*}
    As we can see, for arbitrary pointwise finite-dimensional representations $M$ and $N$, it is possible that $\Ext^1(M,N)$ is infinite-dimensional, even if $M$ and $N$ are both rigid.
\end{example}

\subsection{The left or right $Q$-bounded cases.} 
Without any assumption on the thread quiver $(Q, \mathcal{P})$, we cannot use projective resolutions or injective co-resolutions to study $\rpwf \C$. In this subsection, we impose some restrictions on $Q$ so that projective resolutions or injective co-resolutions can be used. In the left $Q$-bounded case, any indecomposable pointwise finite-dimensional representation $M$ admits an epimorphism $P \to M$ from a pointwise finite-dimensional projective representation. Without the left $Q$-bounded assumption, this is not always possible to find such a morphism. Some dual statements hold in the right $Q$-bounded case.

\begin{proposition}\label{prop:no infinite paths and interval finite implies hereditary}
    Assume that $\C$ is left or right $Q$-bounded. Then $\rpwf \C$ is hereditary.
\end{proposition}

\begin{proof}
    We will prove the left $Q$-bounded case. To obtain the right $Q$-bounded case, consider $\rpwf(\C^{\text{op}})$, since $\C^{\text{op}}$ left $Q$-bounded. Then note that for pointwise-finite-dimensional modules, we have a duality $\rpwf\C\to\rpwf(\C^{\text{op}})$ that respects Ext functors. 
    In particular, $\Ext^n(M,N)=0$ when $n\geq 2$ for any pointwise-finite-dimensional modules $M$ and $N$ in either $\rpwf\C$ or $\rpwf(\C^{\text{op}})$.

    Suppose now that $\C$ is left $Q$-bounded. Since $\rpwf\C$ is Krull--Remak--Schmidt--Azumaya we know that every object is a (unique) direct sum of indecomposable representations.
    Furthermore, we know
    \[\Ext^2\left(\bigoplus_\gamma M_\gamma , N\right) \cong \prod_\gamma \Ext^2 (M_\gamma , N).\]
    Thus, it is sufficient to show that $\Ext^2 (M,N)=0$ whenever $M$ is indecomposable. In order to do so, it suffices to prove that any indecomposable representation $M$ admits a projective resolution of two terms in $\rpwf \C$.

    As shown in the proof of Proposition \ref{ThmAllProj}, for each noise free representation $M$, there is a projective representation $P$ in $\rpwf \C$ with an epimorphism $P \to M$.
    Hence, it suffices to show that if $P \in \rpwf \C$ is projective and $L$ is a subrepresentation of $P$, then $L$ is projective. Let $P \in \rpwf \C$ be projective. It follows from Proposition \ref{ThmAllProj} that all projective indecomposable representations are noise free.
    Hence, $P$ being a direct sum of noise free representations is also noise free, by the locally finite-dimensional property.
    The representation $L$ may not be noise free but must be quasi noise free (Proposition~\ref{thm:quasi noise free is serre}).
    Thus $\mathfrak{P}_L$ is valid. Consider $\mathfrak{P}:=\mathfrak{P}_{P\oplus L}$, having cells given by the intersections of the cells of $\mathfrak{P}_L$ with those of $\mathfrak{P}_P$ (see page~\pageref{partition of the direct sum}).

    Let us now use a set $\mathcal{M}$ of sample points from $\mathfrak{P}$ and consider the functors $\ResM: \rpwf\C \to \rpwf(\C(\mathcal{M}))$ and $\IndP: \rpwf(\C(\mathcal{M})) \to \rpwf\C$ as defined in the previous section on page~\pageref{res and ind}. Recall that $\C(\mathcal{M})$ can be identified with a quiver $\Gamma$ obtained from $Q$ by replacing each arrow by a finite linearly ordered quiver of type $\mathbb{A}$. It is clear that $\ResM$ sends projectives to projectives. Since $\ResM$ is exact, we get a monomorphism $\ResM(L) \to \ResM(P)$ where $\ResM(P)$ is a projective representation of $\Gamma$. Since $\rpwf(\Gamma)$ is hereditary by \cite[Page 79]{GR97}, it follows that $\ResM(L)$ is projective. Now, observe that if $P_x$ denotes the projective representation at vertex $x$ of $\Gamma$, then $\IndP(P_x) = P_I$ where $I$ is the unique interval from the partition $\mathfrak{P}$ containing $x$. Hence, we see that $\IndP$ sends projectives to projectives. By the choice of our partition $\mathfrak{P}$, we see that $\IndP \ResM(L) = L$, hence $L$ is projective as desired.
\end{proof}

The categories of finitely generated (by \emph{all} projectives in $\rpwf\C$) representations and that of finitely presented (by only representable projectives) both have enough projectives.
When $Q$ is finite, the subcategory $\rqnf\C$ of $\rpwf\C$ is generated by finite sums of \emph{all} projectives.
Therefore, by the proof of Proposition~\ref{prop:no infinite paths and interval finite implies hereditary}, these are also hereditary.

We do not know if $\rpwf\C$ is always hereditary.  We recall that thread quivers behave slightly differently than quivers. For instance, we have seen that for some thread quivers, the corresponding category $\Rep \,\C$ of all representations of $\C$ is not hereditary. However, for pointwise-finite-dimensional representations, we believe that the hereditary property should hold. We pose this as a conjecture as follows.

\begin{conjecture}\label{con:pwf hereditary}
    Let $\C$ be the path category of a thread quiver. Then $\rpwf\C$ is hereditary.
\end{conjecture}

\begin{remark}
    Notice that the discrete $\mathbb{A}_\infty$ quiver (without sinks and sources) fails the hypothesis Proposition \ref{prop:no infinite paths and interval finite implies hereditary} because it has an infinite path terminating at a vertex.
    However, we know the category of pointwise finite-dimensional representations of any thread quiver of type $\mathbb{A}_\infty$ is hereditary \cite{IRT22}.
    Thus, the converse of Proposition~\ref{prop:no infinite paths and interval finite implies hereditary} is not true.
\end{remark}

Denote by $\rfp\C$ the full subcategory of $\rpwf\C$ consisting of finitely presented representations (objects generated by only representable projectives).

\begin{proposition} \label{prop:fp_hereditary}
    Assume that $\C$ is left $Q$-bounded. Then the category $\rfp\C$ is hereditary, abelian, and closed under extensions.
\end{proposition}

\begin{proof}
    It is clear that $\rfp\C$ is closed under extension. Hence, the hereditary property follows. To prove that $\rfp\C$ is abelian, it follows from a result of Auslander (see \cite[Prop. 2.1]{A66}) that it suffices to prove that if $f: P_2 \to P_3$ is a morphism of finitely generated projective representations, then there is an exact sequence $P_1 \stackrel{g}{\to} P_2 \stackrel{f}{\to} P_3$ of finitely generated projective representations. First, we consider the image of $f$. Being a subrepresentation of $P_3$, and since $\rpwf \C$ is hereditary, it follows that ${\rm Im}f$ is projective. Now, we consider the short exact sequence
    $$0 \to K \to P_2 \to {\rm Im}f \to 0$$
    where $K$ is the kernel of the projection of $P_2$ onto the image of $f$. This sequence splits which means that $K$ is a direct summand of $P_2$. By the Krull--Remak--Schmidt property, that means that $P_1:=K$ if finitely generated projective, as desired.
\end{proof}



 Note that in \cite{BvR14}, the authors have proven that when the thread quiver $(Q, \mathcal{P})$ is locally discrete, then the category of finitely presented representations of $\C$ is hereditary. In the language of the present paper, locally discrete means that for any arrow $\alpha \in Q_1$ and $x \in \oP$, $x$ has an immediate successor in $\oP$ if $x \ne t(\alpha)$, and has an immediate predecessor in $\oP$ if $x \ne s(\alpha)$.

\begin{example}
In this example, we let $Q$ be a quiver of type $\mathbb{D}_\infty$, that is, an orientation of the following graph.

$$\xymatrix@R=1ex{1 \ar@{-}[dr]^\alpha &&&& \\ & 3 \ar@{-}[r]^{\sigma_1} & 4 \ar@{-}[r] & \cdots \ar@{-}[r]^{\sigma_{n-2}} & \cdots \infty\\
2 \ar@{-}[ur]_\beta &&&&}$$

We assume that $\mathcal{P}_\alpha$ and $\mathcal{P}_\beta$ are both empty while the $\mathcal{P}_{\sigma_i}$ are arbitrary. We consider the path category $\C$ of the corresponding thread quiver. We note that $\C$ is either left $Q$-bounded or right $Q$-bounded. Hence, it follows from Proposition \ref{prop:no infinite paths and interval finite implies hereditary} that $\rpwf \C$ is hereditary.

Observe that for any sample $\M$, the category $\rpwf (\C (\mathcal{M}))$ is equivalent to the path category of a quiver $\Gamma_\M$ which is of type $\mathbb{D}_\infty$. Pointwise-finite indecomposable dimensional representations of a quiver of type $\mathbb{D}_\infty$ are classified; see for example \cite{BLP,Y17,MRZ22}. Such an indecomposable representation is either supported on a finite subquiver of it, or else, the maps are isomorphisms between one-dimensional vector spaces for almost all arrows.

We can then use the $\IndP$ and $\ResM$ functors to classify all our indecomposable pointwise-finite-dimensional representations of $\C$ in this case.
We also remark that $\rpwf \C$ is essentially finite, hence also virtually finite.

\end{example}

\bibliographystyle{alpha}
\bibliography{ref}

\end{document}